\documentclass{birkjour}

\usepackage[latin9]{inputenc}
\usepackage{amssymb}
\usepackage[unicode=true,pdfusetitle,
 bookmarks=true,bookmarksnumbered=false,bookmarksopen=false,
 breaklinks=false,pdfborder={0 0 1},backref=false,colorlinks=false]
 {hyperref}

 \newtheorem{thm}{Theorem}[section]
 \newtheorem{cor}[thm]{Corollary}
 \newtheorem{lem}[thm]{Lemma}
 \newtheorem{prop}[thm]{Proposition}
 \theoremstyle{definition}
 \newtheorem{defn}[thm]{Definition}
 \theoremstyle{remark}

 \numberwithin{equation}{section}

  \theoremstyle{plain}
  \newtheorem*{lem*}{Lemma}
  \theoremstyle{plain}
  \newtheorem*{thm*}{Theorem}

\begin{document}

\title[Eigenfunction Statistics for a Point Scatterer on a 3D 
Torus]{Eigenfunction Statistics for a Point Scatterer on a Three-Dimensional
Torus}

\author{Nadav Yesha}

\begin{abstract}
In this paper we study eigenfunction statistics for a point scatterer
(the Laplacian perturbed by a delta-potential) on a three-dimensional
flat torus. The eigenfunctions of this operator are the eigenfunctions
of the Laplacian which vanish at the scatterer, together with a set
of new eigenfunctions (perturbed eigenfunctions). We first show that
for a point scatterer on the standard torus all of the perturbed eigenfunctions
are uniformly distributed in configuration space. Then we investigate
the same problem for a point scatterer on a flat torus with some irrationality
conditions, and show uniform distribution in configuration space for
almost all of the perturbed eigenfunctions.
\end{abstract}

\address{%
Raymond and Beverly Sackler School of Mathematical Sciences\\
Tel Aviv University\\
Tel Aviv 69978\\
Israel}

\email{nadavye1@post.tau.ac.il}

\maketitle

\section{Introduction}

One of the key results in the field of Quantum Chaos is Schnirelman's
quantum ergodicity theorem \cite{Schnirelman,Colin2,Zelditch}, which
asserts that quantum systems whose classical counterpart have chaotic
dynamics are quantum ergodic, in the sense that for almost all eigenstates
the expectation values of observables converge to the phase space
average, i.e. almost all eigenstates are equidistributed in phase
space. An important case is when \emph{all }expectation values converge
to the phase space average - such behavior is called quantum \emph{unique}
ergodicity.

The opposite of chaotic systems in classical mechanics are systems
with integrable dynamics, whose behavior is predictable over a long
period of time. In this paper we study eigenfunction statistics for
a point scatterer on a three-dimensional flat torus, which is an intermediate
system - its classical dynamics is close to integrable, yet the quantum
system is substantially influenced by the scatterer, and therefore
shares some of the behavior of classically chaotic quantum systems.
We study quantum ergodicity and quantum unique ergodicity in configuration
space (rather than in full phase space), a notion which is of growing
interest in recent research, for example in the field of control theory
(cf. \cite{Privat}).

A point scatterer is formally described by a quantum Hamiltonian 
\begin{equation}
-\Delta+\alpha\delta_{x_{0}}\label{eq:IntroScatterer}
\end{equation}
where $\delta_{x_{0}}$ is the Dirac mass at $x_{0}$ and $\alpha$
is a coupling parameter. Mathematically it is realized as a self-adjoint
extension of the Laplacian $-\Delta$ acting on functions vanishing
near $x_{0}$ (see Section § \ref{sec:Background}). Such extensions
are parametrized by a phase $\phi\in(-\pi,\pi],$ where $\phi=\pi$
corresponds to the standard Laplacian ($\alpha=0$ in \eqref{eq:IntroScatterer}).
For $\phi\neq\pi$, the eigenfunctions of the corresponding operator
consist of eigenfunctions of the Laplacian which vanish at the scatterer,
and new eigenfunctions (perturbed eigenfunctions).

In two dimensions, Rudnick and Ueberschär \cite{Rudnick} proved quantum
ergodicity in configuration space regarding the perturbed eigenfunctions
of a point scatterer on the flat torus $\mathbb{T}^{2}=\mathbb{R}^{2}/2\pi\mathcal{L}_{0}$
where $\mathcal{L}_{0}=\mathbb{Z}\left(1/a,0\right)\oplus\mathbb{Z}\left(a,0\right)$
is (any) unimodular lattice, i.e. they proved that almost all of the
perturbed eigenfunctions are uniformly distributed in configuration
space. Our goal is to prove a similar result for a point scatterer
on the three-dimensional torus, showing uniform distribution in configuration
space for almost all (and hopefully \emph{all}) of the perturbed eigenfunctions.

We remark that in dimensions four and greater, the Laplacian $ -\Delta $ acting on functions vanishing near $ x_0 $ is essentially self-adjoint, so there are no non-trivial self-adjoint extensions in those cases.

The three-dimensional problem provides some essential differences
from the two-dimensional case. For example Weyl's law for the three-dimensional
torus, establishing the asymptotics of the counting function $N\left(x\right)$
of eigenvalues of the Laplacian below $x$, reads as $N\left(x\right)\sim Cx^{3/2}$
for some constant $C$, while in two dimensions we have $N\left(x\right)\sim Cx$,
so we deduce completely different bounds for the density of the perturbed
eigenvalues in each case. Moreover, there are major differences in
the behavior of the eigenvalues of the Laplacian on different three-dimensional
tori (and therefore in the behavior of the perturbed eigenfunctions),
so instead of a general theorem, we will investigate two main cases:
the case of the standard three-dimensional flat torus, and the case
of an irrational torus where the multiplicities of the corresponding
eigenvalues of the Laplacian are bounded.

In the case of the standard torus $\mathbb{T}^{3}=\mathbb{R}^{3}/2\pi\mathbb{Z}^{3}$, the eigenvalues of the Laplacian are the integers which are sums of three squares, and the multiplicity of each eigenvalue is its number of representations as such sums, so we can use some arithmetic properties of sums of three squares and their number of representations to show a stronger result -- we show that for a point scatterer on the standard torus,
\emph{all} of the perturbed eigenfunctions are uniformly distributed
in configuration space. More precisely, for every $\phi\in\left(-\pi,\pi\right)$,
we will have a set of perturbed eigenvalues $\Lambda_{\phi},$ with
the corresponding $L^{2}$-normalized eigenfunctions $g_{\lambda}$
$\left(\lambda\in\Lambda_{\phi}\right)$. We prove the following theorem:
\begin{thm}
\label{thm:Main1}Let $\mathbb{T}^{3}=\mathbb{R}^{3}/2\pi\mathbb{Z}^{3}$ be the standard torus. Fix $\phi\in\left(-\pi,\pi\right).$ Then for all
observables $a\in C^{\infty}\left(\mathbb{T}^{3}\right)$%
\footnote{Consequently, Theorem \ref{thm:Main1} still holds for all observables
which are Riemann integrable on $\mathbb{T}^{3}$. The same is true
for Theorem \ref{thm:Main2}.%
},
\[
\int_{\mathbb{T}^{3}}a\left(x\right)\left|g_{\lambda}\left(x\right)\right|^{2}\mbox{d}x\to\frac{1}{\mbox{area}\left(\mathbb{T}^{3}\right)}\int_{\mathbb{T}^{3}}a\left(x\right)\mbox{d}x
\]
as $\lambda\to\infty$ along $\Lambda_{\phi}$.
\end{thm}
Then we show a similar result for a point scatterer on an irrational
torus, but with convergence only along a density one set in the set
of the perturbed eigenvalues: consider the family of flat tori $\mathbb{T}^{3}=\mathbb{R}^{3}/2\pi\mathcal{L}_{0}$,
where 
\[
\mathcal{L}_{0}=\mathbb{Z}\left(a,0,0\right)\oplus\mathbb{Z}\left(0,b,0\right)\oplus\mathbb{Z}\left(0,0,c\right)
\]
is a lattice such that $1/a^{2},1/b^{2},1/c^{2}\in\mathbb{R}$ are
independent over $\mathbb{Q}$. We also demand that at least one of
the ratios $b^{2}/a^{2},\, c^{2}/a^{2},\, c^{2}/b^{2}$ will be an
irrational of finite type, as in the following definition:
\begin{defn}
\label{def:Type1Irrational}An irrational $\alpha$ is said to be
of finite type $\tau\in\mathbb{R}$, if $\tau$ is the supremum of
all $\gamma$ for which $\underline{\lim}_{q\to\infty}q^{\gamma}\left\|q\alpha\right\|=0,$
where $q$ runs through the positive integers. Here 
\[
\left\|t\right\|=\min_{n\in\mathbb{Z}}\left|t-n\right|=\min\left(\left\{ t\right\} ,\left\{ -t\right\} \right)
\]
denotes the distance from $t$ to the nearest integer.

In particular, if $\alpha$ is an irrational of finite type $\tau,$
then for every $\varepsilon>0$, there exists a positive constant
$c=c\left(\alpha,\varepsilon\right)$ such that $\left\|q\alpha\right\|\geq\frac{c}{q^{\tau+\varepsilon}}$
holds for all positive integers $q$. Also note that by Dirichlet's
Theorem we must have $\tau\geq1$, and every algebraic irrational
is of type $1$ due to the theorem of Roth \cite{Roth}.

As in the case of the standard torus, for every $\phi\in\left(-\pi,\pi\right)$,
we will have a set of perturbed eigenvalues $\Lambda_{\phi},$ with
the corresponding $L^{2}$-normalized eigenfunctions $g_{\lambda}$,
and we prove: \end{defn}
\begin{thm}
\label{thm:Main2}Let $\mathbb{T}^{3}=\mathbb{R}^{3}/2\pi\mathcal{L}_{0}$ be an irrational torus as defined above. Fix $\phi\in\left(-\pi,\pi\right).$ There is a
subset $\Lambda_{\phi,\infty}\subseteq\Lambda_{\phi}$ of density
one so that for all observables $a\in C^{\infty}\left(\mathbb{T}^{3}\right)$,
\[
\int_{\mathbb{T}^{3}}a\left(x\right)\left|g_{\lambda}\left(x\right)\right|^{2}\mbox{d}x\to\frac{1}{\mbox{area}\left(\mathbb{T}^{3}\right)}\int_{\mathbb{T}^{3}}a\left(x\right)\mbox{d}x
\]
as $\lambda\to\infty$ along $\Lambda_{\phi,\infty}$.
\end{thm}

\subsection*{Acknowledgments:}

This work is part of the author\textquoteright{}s M.Sc. thesis written
under the supervision of Zeev Rudnick at Tel Aviv University. Partially
supported by the Israel Science Foundation (grant No. 1083/10).

\section{Background\label{sec:Background}}

\subsection{Point Scatterers on the Flat Torus}

Let $\mathbb{T}^{3}=\mathbb{R}^{3}/2\pi\mathcal{L}_{0}$ be a flat
three-dimensional torus, where 
\[
\mathcal{L}_{0}=\mathbb{Z}\left(a,0,0\right)\oplus\mathbb{Z}\left(0,b,0\right)\oplus\mathbb{Z}\left(0,0,c\right)
\]
is a lattice.

We want to study the Schrödinger operator with a delta-potential on
the flat three-dimensional torus $\mathbb{T}^{3}$, formally given
by
\begin{equation}
-\Delta+\alpha\delta_{x_{0}}\label{eq:scatterer}
\end{equation}
 where $\Delta$ is the associated Laplacian on $\mathbb{T}^{3}$,
$\delta_{x_{0}}$ is the Dirac delta-function at the point $x_{0}$,
and $\alpha\in\mathbb{R}$ is a coupling parameter.

We now give a rigorous mathematical description for the operator \eqref{eq:scatterer}
following \cite{Colin1,Rudnick}:

Consider the domain of $C^{\infty}$-functions which vanish in a neighborhood
of $x_{0}$: 
\[
D_{0}=C_{0}^{\infty}\left(\mathbb{T}^{3}\setminus\left\{ x_{0}\right\} \right)
\]
and denote $-\Delta_{x_{0}}=-\Delta_{|D_{0}}$, which is an operator
in the Hilbert space $L^{2}\left(\mathbb{T}^{3}\right)$. One finds
that the adjoint of $-\Delta_{x_{0}}$ has as its domain $\mbox{Dom}\left(-\Delta_{x_{0}}^{*}\right)$
the Sobolev space $H^{2}\left(\mathbb{T}^{3}\setminus\left\{ x_{0}\right\} \right)$,
which equals the space of $f\in L^{2}\left(\mathbb{T}^{3}\right)$
for which there is some $A\in\mathbb{C}$ such that
\[
-\Delta f+A\delta_{x_{0}}\in L^{2}\left(\mathbb{T}^{3}\right).
\]
For such $f,$ there is some $B\in\mathbb{C}$ so that
\[
f\left(x\right)=A\cdot\frac{-1}{4\pi\left|x-x_{0}\right|}+B+o\left(1\right),\hspace{1em}x\to x_{0}.
\]
One finds that the self-adjoint extensions of $-\Delta_{x_{0}}$ are
parametrized by a phase $\phi\in(-\pi,\pi]$; denoting the corresponding
operators by $-\Delta_{\phi,x_{0}}$, their domain is given by $f\in\mbox{Dom}\left(-\Delta_{x_{0}}^{*}\right)$
for which there is some $a\in\mathbb{C}$ so that 
\[
f\left(x\right)=a\left(\cos\frac{\phi}{2}\cdot\frac{-1}{4\pi\left|x-x_{0}\right|}+\sin\frac{\phi}{2}\right)+o\left(1\right),\hspace{1em}x\to x_{0}.
\]
The action of $-\Delta_{\phi,x_{0}}$ on $f\in\mbox{Dom}\left(-\Delta_{\phi,x_{0}}\right)$
is then given by
\begin{equation}
-\Delta_{\phi,x_{0}}f=-\Delta f+A\delta_{x_{0}}=-\Delta f+a\cos\frac{\phi}{2}\delta_{x_{0}}.\label{eq:ScattererFormal}
\end{equation}
Note that for $\phi=\pi$ we have 
\[
\mbox{Dom}\left(-\Delta_{\pi,x_{0}}\right)=H^{2}\left(\mathbb{T}^{3}\right)=\left\{ f\in L^{2}\left(\mathbb{T}^{3}\right):\,-\Delta f\in L^{2}\left(\mathbb{T}^{3}\right)\right\} 
\]
 and 
\[
-\Delta_{\pi,x_{0}}f=-\Delta f
\]
so this extension retrieves the standard Laplacian $-\Delta_{\infty}$
on the domain $H^{2}\left(\mathbb{T}^{3}\right)$ (which is the unique
self-adjoint extension of $-\Delta_{|C^{\infty}\left(\mathbb{T}^{3}\right)})$. 

The operator $-\Delta_{\infty}$ has a discrete spectrum; an orthonormal
basis of eigenfunctions for $-\Delta_{\infty}$ consists of the functions
\[
\frac{1}{\sqrt{\mbox{area}\left(\mathbb{T}^3\right)}}e_{\xi}
\]
where
\[
e_{\xi}=\exp\left(i\xi\cdot\left(x-x_{0}\right)\right)
\]
and $\xi$ ranges over the dual lattice
\[
\mathcal{L}=\left\{ \xi\in\mathbb{R}^{3}:\,\xi\cdot l\in\mathbb{Z}\hspace{1em}\forall l\in\mathcal{L}_{0}\right\} =\mathbb{Z}\left(\frac{1}{a},0,0\right)\oplus\mathbb{Z}\left(0,\frac{1}{b},0\right)\oplus\mathbb{Z}\left(0,0,\frac{1}{c}\right).
\]
The corresponding eigenvalues are the norms $\left|\xi\right|^{2}$
of the vectors of the dual lattice $\mathcal{L}$; denote by $\mathcal{N}$
the set of these norms. In the case of the standard torus $\mathcal{L}_{0}=\mathbb{Z}^{3}$
(and then $\mathcal{L}=\mathbb{Z}^{3})$ we have $\mathcal{N}=\mathcal{N}_{3}$,
where $\mathcal{N}_{3}$ is the set of integers which are sums of
three squares, and each eigenvalue is of multiplicity $r_{3}\left(n\right)$
which is the number of representations of $n=a^{2}+b^{2}+c^{2}$ with
$a,b,c\in\mathbb{Z}$ integers.

For the perturbed operator \eqref{eq:ScattererFormal} with $\phi\neq\pi$
we still have the nonzero eigenvalues from the unperturbed problem
$\left(0\neq\lambda\in\sigma\left(-\Delta_{\infty}\right)\right)$,
with multiplicity decreased by one, as well as a new set $\Lambda=\Lambda_{\phi}$
of eigenvalues, each appearing with multiplicity one, which are the
solutions to the equation
\begin{equation}
\sum_{\xi\in\mathcal{L}}\left\{ \frac{1}{\left|\xi\right|^{2}-\lambda}-\frac{\left|\xi\right|^{2}}{\left|\xi\right|^{4}+1}\right\} =c_{0}\tan\frac{\phi}{2}\label{eq:eigenvalues}
\end{equation}
where
\[
c_{0}=\sum_{\xi\in\mathcal{L}}\frac{1}{\left|\xi\right|^{4}+1}
\]
with the corresponding eigenfunctions being multiples of the Green's
function 
\[
G_{\lambda}\left(x,x_{0}\right)=\left(\Delta+\lambda\right)^{-1}\delta_{x_{0}}
\]
which is an element of $\mbox{Dom}\left(-\Delta_{x_{0}}^{*}\right)$
for every $\lambda\notin\sigma\left(-\Delta_{\infty}\right)$, and
has the $L^{2}$-expansion
\[
G_{\lambda}(x,x_{0})=-\frac{1}{8\pi^{3}}\sum\limits _{\xi\in\mathcal{L}}\frac{\exp\left(i\xi\cdot\left(x-x_{0}\right)\right)}{\left|\xi\right|^{2}-\lambda}.
\]
\eqref{eq:eigenvalues} can be written as
\[
\sum_{n\in\mathcal{N}}r_{\mathcal{L}}\left(n\right)\left\{ \frac{1}{n-\lambda}-\frac{n}{n^{2}+1}\right\} =c_{0}\tan\frac{\phi}{2}
\]
where
\[
r_{\mathcal{L}}\left(n\right)=\#\left\{ \xi\in\mathcal{L}:\,\left|\xi\right|^{2}=n\right\} 
\]
is the multiplicity of the norm $n$. The function 
\[
F\left(\lambda\right)=\sum_{n\in\mathcal{N}}r_{\mathcal{L}}\left(n\right)\left\{ \frac{1}{n-\lambda}-\frac{n}{n^{2}+1}\right\} 
\]
is meromorphic with simple poles in $n\in\mathcal{N}$, and $F_{|\mathbb{R}}$
is strictly increasing between the poles, so if we label
\[
\mathcal{N}=\left\{ 0=n_{0}<n_{1}<n_{2}<\dots\right\} 
\]
then the new eigenvalues interlace between the elements of $\mathcal{N}$,
and we may denote the perturbed eigenvalues by $\lambda_{k}=\lambda_{k}^{\phi}$
so that
\[
\lambda_{0}<n_{0}<\lambda_{1}<n_{1}<\lambda_{2}<\dots<n_{k}<\lambda_{k+1}<n_{k+1}<\dots.
\]
We say that a subset $\Lambda'=\left\{ \lambda_{j_{k}}\right\} \subseteq\Lambda$
is of density $a$ ($0\leq a\leq1)$ in $\Lambda$ if 
\[
\lim_{J\to\infty}\frac{1}{J}\#\left\{ k\in\mathbb{N}:\, j_{k}\leq J\right\} =a
\]
or equivalently
\[
\lim_{X\to\infty}\frac{\#\left\{ \lambda\in\Lambda':\,\lambda\leq X\right\} }{\#\left\{ \lambda\in\Lambda:\,\lambda\leq X\right\} }=a.
\]
Denote by 
\[
g_{\lambda}\left(x\right):=\frac{G_{\lambda}\left(x,x_{0}\right)}{\left\|G_{\lambda}\right\|_{2}}
\]
the $L^{2}$-normalized Green's function.

\subsection{Arithmetic Background}

In this section we recall some basic arithmetic facts, that we will
need to use in the proof of Theorem \ref{thm:Main1} for the standard
torus.

By the famous theorem due to Legendre and Gauss (see \cite{Grosswald}
for example), the Diophantine equation
\begin{equation}
x_{1}^{2}+x_{2}^{2}+x_{3}^{2}=n\label{eq:3squares}
\end{equation}
has solutions in integers $x_{i}$ $\left(i=1,2,3\right)$ if and
only if $n$ is not of the form $4^{a}\left(8k+7\right)$ with $a\in\mathbb{Z}$,
$a\geq0$ and $k\in\mathbb{Z}$. Denote by $r_{3}\left(n\right)$
the number of solutions to \eqref{eq:3squares}, then for all $n$,
$r_{3}\left(4^{a}n\right)=r_{3}\left(n\right).$

Equivalently, if we write $n=4^{a}n_{1},$ with $4\nmid n_{1}$, then
$n$ is a sum of three squares if and only if $n_{1}\not\equiv7\,\left(8\right)$,
that is to say
\[
\mathcal{N}_{3}=\left\{ n\in\mathbb{N}:\, n=4^{a}n_{1},\,4\nmid n_{1}\Rightarrow n_{1}\not\equiv7\,\left(8\right)\right\} ,
\]
and $r_{3}\left(n\right)=r_{3}\left(n_{1}\right)$.

The fact that for all $n$: $r_{3}\left(4^{a}n\right)=r_{3}\left(n\right)$,
follows from a simple lemma, that will be in use for us later:
\begin{lem}
\label{lem:4Powers}For every $\xi\in\mathbb{Z}^{3}$ and $a\geq0$,
\[
4^{a}\mid\left|\xi\right|^{2}\Longleftrightarrow\xi=2^{a}\xi_{1}\,\left(\xi_{1}\in\mathbb{Z}^{3}\right).
\]
\end{lem}
\begin{proof}
If $\xi=2^{a}\xi_{1}$ and $\xi_{1}\in\mathbb{Z}^{3}$, then $\left|\xi\right|^{2}=4^{a}\left|\xi_{1}\right|^{2}$,
so $4^{a}\mid\left|\xi\right|^{2}$.

The other direction is proved by induction on $a$: the case $a=0$
is clear. Assume that the argument is true for $a$, and that $4^{a+1}\mid\left|\xi\right|^{2}$.
Denoting by $\xi=\left(x,y,z\right)$ we get in particular that $x^{2}+y^{2}+z^{2}\equiv0\,\left(4\right)$.
Since clearly $x^{2},y^{2},z^{2}\equiv0,1\,\left(4\right)$, it follows
that necessarily $x^{2},y^{2},z^{2}\equiv0\,\left(4\right)$, so $x,y$
and $z$ are all even. If we write $x=2x_{1},\, y=2y_{1},\, z=2z_{1}$,
and define $\xi_{0}=\left(x_{1},y_{1},z_{1}\right)\in\mathbb{Z}^{3}$,
we get that $4^{a+1}\mid\left|\xi\right|^{2}=\left|2\xi_{0}\right|^{2}=4\left|\xi_{0}\right|^{2}$,
so $4^{a}\mid\left|\xi_{0}\right|^{2}$. From the induction hypothesis
$\xi_{0}=2^{a}\xi_{1}\,\left(\xi_{1}\in\mathbb{Z}^{3}\right)$, and
we get that $\xi=2\xi_{0}=2^{a+1}\xi_{1}$.
\end{proof}
Denote by $R_{3}\left(n\right)$ the number of primitive solutions
to \eqref{eq:3squares}, i.e. the number of solutions such that $\gcd\left(x_{1},x_{2},x_{3}\right)=1$,
then we have
\begin{equation}
r_{3}\left(n\right)=\sum_{d^{2}\mid n}R_{3}\left(\frac{n}{d^{2}}\right).\label{eq:r3n}
\end{equation}
We will need some asymptotic bounds for $r_{3}\left(n\right)$. For
an upper bound, assume that $n$ is a sum of three squares, and as
before, write $n=4^{a}n_{1}$ with $4\nmid n_{1}$, so $n_{1}\not\equiv0,4,7\,\left(8\right)$.
We will use the following theorem of Gauss (see \cite{Grosswald}):
\begin{equation}
R_{3}\left(n\right)=\pi^{-1}G_{n}\sqrt{n}L\left(1,\chi\right)\label{eq:R3n}
\end{equation}
with 
\[
G_{n}=\begin{cases}
0 & n\equiv0,4,7\,\left(8\right)\\
16 & n\equiv3\,\left(8\right)\\
24 & n\equiv1,2,5,6\,\left(8\right)
\end{cases}
\]
where $L\left(1,\chi\right)=\sum\limits _{m=1}^{\infty}\frac{\chi\left(m\right)}{m}$,
and $\chi\left(m\right)=\left(\frac{-4n}{m}\right)$ (the Kronecker
symbol, so $\chi$ is a quadratic character modulo $4n$).

From \eqref{eq:R3n} we have $R_{3}\left(n_{1}\right)\asymp\sqrt{n_{1}}L\left(1,\chi\right)$
(here $\chi\left(m\right)=\left(\frac{-4n_{1}}{m}\right)$). To bound
$L\left(1,\chi\right)$ from above, write
\[
L\left(1,\chi\right)=\sum_{m=1}^{\infty}\frac{\chi\left(m\right)}{m}=\sum_{m=1}^{4n_{1}}\frac{\chi\left(m\right)}{m}+\sum_{4n_{1}+1}^{\infty}\frac{\chi\left(m\right)}{m}.
\]
Clearly 
\[
\left|\sum_{m=1}^{4n_{1}}\frac{\chi\left(m\right)}{m}\right|\leq\sum_{m=1}^{4n_{1}}\frac{1}{m}\ll\log n_{1}
\]
and for the second sum, summation by parts yields
\[
\sum_{4n_{1}+1}^{\infty}\frac{\chi\left(m\right)}{m}\ll\int_{4n_{1}}^{\infty}\frac{s\left(t\right)}{t^{2}}\mbox{d}t
\]
where $s\left(t\right)=\sum\limits _{k\leq t}\chi\left(k\right)$.
But $\left|s\left(t\right)\right|\leq4n_{1}$, so 
\[
\left|\sum_{4n_{1}+1}^{\infty}\frac{\chi\left(m\right)}{m}\right|\ll4n_{1}\int_{4n_{1}}^{\infty}\frac{\mbox{d}t}{t^{2}}=1
\]
and we conclude that $\left|L\left(1,\chi\right)\right|\ll\log n_{1}$,
so $R_{3}\left(n_{1}\right)\ll\sqrt{n_{1}}\log n_{1}$. Note that
if $d^{2}\mid n_{1}$, then $\frac{n_{1}}{d^{2}}\not\equiv0,4,7\left(8\right)$,
so
\[
R_{3}\left(\frac{n_{1}}{d^{2}}\right)\ll\sqrt{\frac{n_{1}}{d}}\log\left(\frac{n_{1}}{d}\right)\leq\sqrt{n_{1}}\log n_{1}
\]
 and using \eqref{eq:r3n} we get that
\begin{align*}
r_{3}\left(n\right)=r_{3}\left(n_{1}\right)&=\sum_{d^{2}\mid n_{1}}R_{3}\left(\frac{n_{1}}{d^{2}}\right)\\
&\ll n_{1}^{1/2}\log n_{1}\sum_{d^{2}\mid n_{1}}1\\
&\leq n_{1}^{1/2}\log n_{1}\sum_{d\mid n_{1}}1\\
&\ll_{\varepsilon}n_{1}^{1/2+\varepsilon}\\
&\leq n^{1/2+\varepsilon}.
\end{align*}
We cannot have a lower bound for $r_{3}\left(n\right)$ for every
$n$, so assume now that 
\[
n\not\equiv0,4,7\,\left(8\right).
\]
Again, from \eqref{eq:R3n} we have
\[
r_{3}\left(n\right)\geq R_{3}\left(n\right)\asymp\sqrt{n}L\left(1,\chi\right)
\]
and by Siegel's theorem \cite{Siegel}: $L\left(1,\chi\right)\gg_{\varepsilon}n^{-\varepsilon}$,
so 
\[
r_{3}\left(n\right)\gg_{\varepsilon}n^{1/2-\varepsilon}.
\]

\section{The Standard Torus}

\subsection{Bounds for the Green's Function and Truncation}

We begin with the proof of Theorem \ref{thm:Main1}, so here $\mathbb{T}^{3}=\mathbb{R}^{3}/2\pi\mathbb{Z}^{3}$,
$\mathcal{L}=\mathbb{Z}^{3}$, and $\mathcal{N}=\mathcal{N}_{3}$.

We first want to give a lower bound for the $L^{2}$-norm of the Green's
function $G_{\lambda}$:
\begin{lem}
\label{lem:RationalNormBound}For every $\lambda\in\Lambda$, we have
\[
\left\|G_{\lambda}\right\|_{2}^{2}\gg\lambda^{1/2-\varepsilon}.
\]
\end{lem}
\begin{proof}
Note that
\[
\left\|G_{\lambda}\right\|_{2}^{2}\asymp\sum_{\xi\in\mathbb{Z}^{3}}\frac{1}{\left(\left|\xi\right|^{2}-\lambda\right)^{2}}=\sum_{n=0}^{\infty}\frac{r_{3}\left(n\right)}{\left(n-\lambda\right)^{2}}.
\]
Take $n_{0}>\lambda$, $n_{0}\equiv1\,\left(8\right)$, $n_{0}-\lambda\leq10,$
then 
\[
\sum_{n=0}^{\infty}\frac{r_{3}\left(n\right)}{\left(n-\lambda\right)^{2}}\geq\frac{r_{3}\left(n_{0}\right)}{\left(n_{0}-\lambda\right)^{2}}\gg r_{3}\left(n_{0}\right)\gg n_{0}^{1/2-\varepsilon}>\lambda^{1/2-\varepsilon}.\qedhere
\]

\end{proof}
We will now use a truncation procedure.

For $L>0$, denote by
\[
G_{\lambda,L}=-\frac{1}{8\pi^{3}}\sum\limits _{\left|\left|\xi\right|^{2}-\lambda\right|<L}\frac{\exp\left(i\xi\cdot\left(x-x_{0}\right)\right)}{\left|\xi\right|^{2}-\lambda}
\]
the truncated Green's function, and let $g_{\lambda,L}$ be the $L^{2}$-normalized
truncated Green's function:
\[
g_{\lambda,L}=\frac{G_{\lambda,L}}{\left\|G_{\lambda,L}\right\|_{2}}.
\]
We have the following approximation:
\begin{lem}
\label{lem:RationalTruncConv}Let $L=\lambda^{\delta},\,0<\delta<1/4$.
Then $\left\|g_{\lambda}-g_{\lambda,L}\right\|_{2}\to0$
as $\lambda\to\infty$.\end{lem}
\begin{proof}
Clearly
\begin{align}
\left\|g_{\lambda}-g_{\lambda,L}\right\|_{2}&=\left\|\frac{G_{\lambda}}{\left\|G_{\lambda}\right\|_{2}}-\frac{G_{\lambda,L}}{\left\|G_{\lambda,L}\right\|_{2}}\right\|_{2}\label{eq:NormalizedDiffsEst}\\
&\leq\frac{\left\|G_{\lambda}-G_{\lambda,L}\right\|_{2}}{\left\|G_{\lambda}\right\|_{2}}+\left\|G_{\lambda,L}\right\|_{2}\left|\frac{1}{\left\|G_{\lambda}\right\|_{2}}-\frac{1}{\left\|G_{\lambda,L}\right\|_{2}}\right|\nonumber\\
&\leq2\frac{\left\|G_{\lambda}-G_{\lambda,L}\right\|_{2}}{\left\|G_{\lambda}\right\|_{2}}.\nonumber 
\end{align}
We have
\begin{align*}
\left\|G_{\lambda}-G_{\lambda,L}\right\|_{2}^{2}\asymp\sum_{\left|n-\lambda\right|\geq\lambda^{\delta}}\frac{r_{3}\left(n\right)}{\left(n-\lambda\right)^{2}}&\ll\sum_{\left|n-\lambda\right|\geq\lambda^{\delta}}\frac{n^{1/2+\varepsilon}}{\left(n-\lambda\right)^{2}}\\
&\ll\int_{\begin{subarray}{c}
\left|x-\lambda\right|\geq\frac{1}{2}\lambda^{\delta}\\
x\geq0
\end{subarray}}\frac{x^{1/2+\varepsilon}\mbox{d}x}{\left(x-\lambda\right)^{2}}\\
&=\lambda^{-1/2+\varepsilon}\int_{\begin{subarray}{c}
\left|y-1\right|\geq\frac{1}{2}\lambda^{-1+\delta}\\
y\geq0
\end{subarray}}\frac{y^{1/2+\varepsilon}\mbox{d}y}{\left(y-1\right)^{2}}\\
&\ll\lambda^{-\delta+1/2+\varepsilon}.
\end{align*}
Using Lemma \ref{lem:RationalNormBound} we conclude that 
\begin{equation}
\frac{\left\|G_{\lambda}-G_{\lambda,L}\right\|_{2}^{2}}{\left\|G_{\lambda}\right\|_{2}^{2}}\ll\frac{\lambda^{-\delta+1/2+\varepsilon}}{\lambda^{1/2-\varepsilon}}=\lambda^{-\delta+2\varepsilon}\label{eq:RationalTruncDiff}
\end{equation}
which tends to zero (for $\varepsilon>0$ small enough) since $\delta>0$.
\end{proof}
We conclude that the $L^{2}$-norm of the truncated Green's function
$G_{\lambda,L}$ is asymptotically equivalent to the $L^{2}$-norm
of the non-truncated function $G_{\lambda}$:
\begin{lem}
\label{lem:RationalNormEquiv}Let $L=\lambda^{\delta},\,0<\delta<1/4$.
Then 
\[
\left\|G_{\lambda,L}\right\|_{2}=\left\|G_{\lambda}\right\|_{2}\left(1+o\left(1\right)\right).
\]
\end{lem}
\begin{proof}
This follows from \eqref{eq:RationalTruncDiff}, since 
\[
\frac{\left|\left\|G_{\lambda,L}\right\|_{2}-\left\|G_{\lambda}\right\|_{2}\right|}{\left\|G_{\lambda}\right\|_{2}}\leq\frac{\left\|G_{\lambda}-G_{\lambda,L}\right\|_{2}}{\left\|G_{\lambda}\right\|_{2}}\to0
\]
as $\lambda\to\infty$.
\end{proof}
We turn to prove the next approximation:
\begin{lem}
\label{lem:RationalSuffConv}Let $L=\lambda^{\delta},\,0<\delta<1/4$.
For every $f\in C^{\infty}\left(\mathbb{T}^{3}\right)$, we have 
\[
\left|\left\langle fg_{\lambda},g_{\lambda}\right\rangle -\left\langle fg_{\lambda,L},g_{\lambda,L}\right\rangle \right|\to0
\]
as $\lambda\to\infty$, so
\[
\left\langle fg_{\lambda,L},g_{\lambda,L}\right\rangle \to0\,\Rightarrow\left\langle fg_{\lambda},g_{\lambda}\right\rangle \to0
\]
as $\lambda\to\infty$.\end{lem}
\begin{proof}
Let $f\in C^{\infty}\left(\mathbb{T}^{3}\right).$ We have
\[
\left|\left\langle fg_{\lambda},g_{\lambda}\right\rangle -\left\langle fg_{\lambda,L},g_{\lambda,L}\right\rangle \right|\leq\left|\left\langle fg_{\lambda},g_{\lambda}-g_{\lambda,L}\right\rangle \right|+\left|\left\langle f\left(g_{\lambda}-g_{\lambda,L}\right),g_{\lambda,L}\right\rangle \right|.
\]
The Cauchy-Schwarz inequality implies that
\[
\left|\left\langle fg_{\lambda},g_{\lambda}-g_{\lambda,L}\right\rangle \right|\leq\left\|fg_{\lambda}\right\|_{2}\left\|g_{\lambda}-g_{\lambda,L}\right\|_{2}\leq\left\|f\right\|_{\infty}\left\|g_{\lambda}-g_{\lambda,L}\right\|_{2}.
\]
From the same reason 
\[
\left|\left\langle f\left(g_{\lambda}-g_{\lambda,L}\right),g_{\lambda,L}\right\rangle \right|\leq\left\|f\left(g_{\lambda}-g_{\lambda,L}\right)\right\|_{2}\le\left\|f\right\|_{\infty}\left\|g_{\lambda}-g_{\lambda,L}\right\|_{2},
\]
but by Lemma \ref{lem:RationalTruncConv} we know that 
\[
\left\|g_{\lambda}-g_{\lambda,L}\right\|_{2}\to0
\]
so
\[
\left|\left\langle fg_{\lambda},g_{\lambda}\right\rangle -\left\langle fg_{\lambda,L},g_{\lambda,L}\right\rangle \right|\to0
\]
as $\lambda\to\infty$. It follows that if we have $\left\langle fg_{\lambda,L},g_{\lambda,L}\right\rangle \to0$,
then 
\[
\left|\left\langle fg_{\lambda},g_{\lambda}\right\rangle \right|\leq\left|\left\langle fg_{\lambda,L},g_{\lambda,L}\right\rangle \right|+\left|\left\langle fg_{\lambda},g_{\lambda}\right\rangle -\left\langle fg_{\lambda,L},g_{\lambda,L}\right\rangle \right|\to0.
\]
so
\[
\left\langle fg_{\lambda},g_{\lambda}\right\rangle \to0
\]
as $\lambda\to\infty$.
\end{proof}

\subsection{Powers of 4}

We want to divide the elements of $\mathcal{N}_{3}$ into two kinds:
those which are divisible by a high power of $4$, and those which
are not.

Fix $0\neq\zeta\in\mathbb{Z}^{3}$, and write $\left|\zeta\right|^{2}=n_{\zeta}=4^{a_{\zeta}}n_{1}^{\zeta}$,
with $4\nmid n_{1}^{\zeta}$. 

We make the following definition:
\begin{defn}
Define 
\[
\mathcal{N}_{0}^{\zeta}=\left\{ n\in\mathcal{N}_{3}:\, n=4^{a}n_{1},\,4\nmid n_{1}\,\Rightarrow a>a_{\zeta}\right\} ,
\]
the set of elements which are divisible by a high power of $4$, and
define 
\[
\mathcal{N}_{1}^{\zeta}=\left\{ n\in\mathcal{N}_{3}:\, n=4^{a}n_{1},\,4\nmid n_{1}\,\Rightarrow a\leq a_{\zeta}\right\} ,
\]
the complement set in $\mathcal{N}_{3}$.
\end{defn}
The following observation will be useful:
\begin{lem}
\label{lem:RationalLow4Powers}For every $\xi\in\mathbb{Z}_{3}$,
if $2\left\langle \xi,\zeta\right\rangle =\left|\zeta\right|^{2}$,
then $\left|\xi\right|^{2}\in\mathcal{N}_{1}^{\zeta}$.\end{lem}
\begin{proof}
Let $\xi\in\mathbb{Z}^{3}$ such that $2\left\langle \xi,\zeta\right\rangle =\left|\zeta\right|^{2}$
, and write $\left|\xi\right|^{2}=n=4^{a}n_{1},$ with $4\nmid n_{1}$.
By Lemma \ref{lem:4Powers}, $\xi=2^{a}\xi_{1}$, with $\left|\xi_{1}\right|^{2}=n_{1}$,
and $\zeta=2^{a_{\zeta}}\zeta_{1}$, with $\left|\zeta_{1}\right|^{2}=n_{1}^{\zeta}$.
Therefore we get that 
\[
2^{a+a_{\zeta}+1}\left\langle \xi_{1},\zeta_{1}\right\rangle =\left|\zeta\right|^{2}=4^{a_{\zeta}}n_{1}^{\zeta}
\]
so
\[
2^{a-a_{\zeta}+1}\left\langle \xi_{1},\zeta_{1}\right\rangle =n_{1}^{\zeta}
\]
and since $4\nmid n_{1}^{\zeta},$ we get that $a-a_{\zeta}+1\leq1,$
so $a\leq a_{\zeta}$, and $\left|\xi\right|^{2}\in\mathcal{N}_{1}^{\zeta}$.\end{proof}
\begin{cor}
\label{cor:RationalHigh4Powers}For every $\xi\in\mathbb{Z}_{3}$,
if $\left|\xi\right|^{2}\in\mathcal{N}_{0}^{\zeta}$, then $\left|2\left\langle \xi,\zeta\right\rangle -\left|\zeta\right|^{2}\right|\geq1$.\end{cor}
\begin{proof}
Let $\xi\in\mathbb{Z}^{3}$ such that $\left|\xi\right|^{2}\in\mathcal{N}_{0}^{\zeta}$.
From Lemma \ref{lem:RationalLow4Powers} we have $2\left\langle \xi,\zeta\right\rangle -\left|\zeta\right|^{2}\neq0$,
and since $2\left\langle \xi,\zeta\right\rangle -\left|\zeta\right|^{2}$
is an integer, we get that $\left|2\left\langle \xi,\zeta\right\rangle -\left|\zeta\right|^{2}\right|\geq1.$
\end{proof}
For every $\lambda\in\Lambda$, define $n_{\lambda}$ to be the element
of $\mathcal{N}_{3}$ which is closest to $\lambda$ (if there are
two elements with the same distance from $\lambda$, take the smallest
of them). Note that since the elements of $ \Lambda $ interlace between the elements of $ \mathcal{N}_3 $, and since for every $n\not\equiv 0,4,7\,\left(8\right)$ we have $ n\in\mathcal{N}_3 $, we conclude that for every $ \lambda\in\Lambda $ we have $ \left|n_\lambda-\lambda\right|\le 1.5 $, and in particular $ n_\lambda \sim \lambda $.

We conclude this section with the following lemma:
\begin{lem}
\label{lem:Rational4PowersDistance}Assume that $n_{\lambda}\in\mathcal{N}_{0}^{\zeta}.$
Then for every $\xi\in\mathbb{Z}^{3}$:
\[
\left|\left|\xi\right|^{2}-\lambda\right|<\frac{1}{2}\,\Longrightarrow\left|\left|\xi-\zeta\right|^{2}-\lambda\right|>\frac{1}{2}.
\]
 \end{lem}
\begin{proof}
Let $\xi\in\mathbb{Z}^{3}$, and assume that $\left|\left|\xi\right|^{2}-\lambda\right|<\frac{1}{2}$.\\
It clearly follows that $\left|\xi\right|^{2}=n_{\lambda}\in\mathcal{N}_{0}^{\zeta}.$
By Corollary \ref{cor:RationalHigh4Powers} we get that:
\begin{align*}
\left|\left|\xi-\zeta\right|^{2}-\lambda\right|&=\left|\left|\xi\right|^{2}-\lambda-2\left\langle \xi,\zeta\right\rangle +\left|\zeta\right|^{2}\right|\\
&\geq\left|2\left\langle \xi,\zeta\right\rangle -\left|\zeta\right|^{2}\right|-\left|\left|\xi\right|^{2}-\lambda\right|>\frac{1}{2}.\qedhere
\end{align*}

\end{proof}

\subsection{Proof of Theorem \ref{thm:Main1}}

We are now in condition to prove Theorem \ref{thm:Main1}. We will
need an estimate for the number of integral points inside some strips
on three-dimensional spheres:
\begin{lem*}
Let $L=\lambda^{\delta},\,0<\delta<1/4$. For every $0\neq\zeta\in\mathbb{Z}^{3},C_{1},C_{2}$
and $n$ such that $\left|n-\lambda\right|<C_{1}L$, we have 
\[
\#\left\{ \eta\in\mathbb{Z}^{3}:\,\left|\eta\right|^{2}=n,\,\left|\left\langle \eta,\zeta\right\rangle \right|<C_{2}L\right\} \ll_{C_{1},C_{2},\zeta,\varepsilon}Ln^{\varepsilon}.
\]

\end{lem*}
This is Lemma \ref{lem:IntSphericalStripsEst} in the Appendix, see
there for a proof.

The following main proposition will easily imply Theorem \ref{thm:Main1}:
\begin{prop}
\label{prop:MainProp}For every $0\neq\zeta\in\mathbb{Z}^{3}$, we
have
\[
\left\langle e_{\zeta}g_{\lambda},g_{\lambda}\right\rangle \to0
\]
as $\lambda\to\infty$.\end{prop}
\begin{proof}
Let $L=\lambda^{\delta},\,0<\delta<1/4$. By Lemma \ref{lem:RationalSuffConv},
it suffices to show that 
\[
\left\langle e_{\zeta}g_{\lambda,L},g_{\lambda,L}\right\rangle \to0
\]
as $\lambda\to\infty$. Note that
\[
\left\langle e_{\zeta}G_{\lambda,L},G_{\lambda,L}\right\rangle \asymp\sum\limits _{\left|\left|\xi\right|^{2}-\lambda\right|<L}\frac{1}{\left(\left|\xi-\zeta\right|^{2}-\lambda\right)\left(\left|\xi\right|^{2}-\lambda\right)}
\]
and therefore
\begin{align*}
\left|\left\langle e_{\zeta}G_{\lambda,L},G_{\lambda,L}\right\rangle \right|&\ll\sum\limits _{\left|\left|\xi\right|^{2}-\lambda\right|<L}\frac{1}{\left|\left|\xi-\zeta\right|^{2}-\lambda\right|\left|\left|\xi\right|^{2}-\lambda\right|}\\
&=\sum\nolimits ^{1}\frac{1}{\left|\left|\xi-\zeta\right|^{2}-\lambda\right|\left|\left|\xi\right|^{2}-\lambda\right|}\\
&+\sum\nolimits ^{2}\frac{1}{\left|\left|\xi-\zeta\right|^{2}-\lambda\right|\left|\left|\xi\right|^{2}-\lambda\right|}
\end{align*}
where in $\sum^{1}$ the summation is over $\xi\in\mathbb{Z}^{3}$
such that: 
\[
\left|\left|\xi\right|^{2}-\lambda\right|<L,\,\left|\left\langle \xi,\zeta\right\rangle \right|\geq L,
\]
and in $\sum^{2}$ the summation is over $\xi\in\mathbb{Z}^{3}$ such
that: 
\[
\left|\left|\xi\right|^{2}-\lambda\right|<L,\,\left|\left\langle \xi,\zeta\right\rangle \right|<L.
\]
Note that
\begin{align*}
\left|\left|\xi-\zeta\right|^{2}-\lambda\right|&=\left|\left|\xi\right|^{2}-\lambda-2\left\langle \xi,\zeta\right\rangle +\left|\zeta\right|^{2}\right|\\
&\geq2\left|\left\langle \xi,\zeta\right\rangle \right|-\left|\left|\xi\right|^{2}-\lambda\right|-\left|\zeta\right|^{2}
\end{align*}
so if $\left|\left|\xi\right|^{2}-\lambda\right|<L,\,\left|\left\langle \xi,\zeta\right\rangle \right|\geq L$,
then 
\[
\left|\left|\xi-\zeta\right|^{2}-\lambda\right|\geq2L-L-\left|\zeta\right|^{2}\gg L,
\]
and hence
\begin{align*}
\sum\nolimits ^{1}\frac{1}{\left|\left|\xi-\zeta\right|^{2}-\lambda\right|\left|\left|\xi\right|^{2}-\lambda\right|}&\ll\frac{1}{L}\sum\nolimits ^{1}\frac{1}{\left|\left|\xi\right|^{2}-\lambda\right|}\\
&\leq\frac{1}{L}\sum\limits _{\left|\left|\xi\right|^{2}-\lambda\right|<L}\frac{1}{\left|\left|\xi\right|^{2}-\lambda\right|}.
\end{align*}
Cauchy-Schwarz gives
\begin{align*}
\sum\limits _{\left|\left|\xi\right|^{2}-\lambda\right|<L}\frac{1}{\left|\left|\xi\right|^{2}-\lambda\right|}&\ll\left\|G_{\lambda,L}\right\|_{2}\left(\sum_{\left|\left|\xi\right|^{2}-\lambda\right|<L}1\right)^{1/2}\\
&=\left\|G_{\lambda,L}\right\|_{2}\left(\sum_{\left|n-\lambda\right|<L}r_{3}\left(n\right)\right)^{1/2}\\
&\ll\left\|G_{\lambda,L}\right\|_{2}\left(\sum_{\left|n-\lambda\right|<L}n^{1/2+\varepsilon}\right)^{1/2}\\
&\ll\left\|G_{\lambda,L}\right\|_{2}L^{1/2}\lambda^{1/4+\varepsilon/2}
\end{align*}
and therefore
\begin{align*}
\sum\nolimits ^{1}\frac{1}{\left|\left|\xi-\zeta\right|^{2}-\lambda\right|\left|\left|\xi\right|^{2}-\lambda\right|}&\ll\frac{\left\|G_{\lambda,L}\right\|_{2}\lambda^{1/4+\varepsilon/2}}{L^{1/2}}\\
&=\left\|G_{\lambda,L}\right\|_{2}\lambda^{-\delta/2+1/4+\varepsilon/2}\\
&\asymp\left\|G_{\lambda}\right\|_{2}\lambda^{-\delta/2+1/4+\varepsilon/2}.
\end{align*}
For the estimation of $\sum^{2}$, remember that we defined $n_{\lambda}$
to be the element of $\mathcal{N}_{3}$ which is closest to $\lambda$
(and if there are two elements with the same distance from $\lambda$,
we take $n_{\lambda}$ to be the smallest of them). We distinguish
between two cases: whether $n_{\lambda}\in\mathcal{N}_{0}^{\zeta}$
or $n_{\lambda}\in\mathcal{N}_{1}^{\zeta}$.

First, assume that $n_{\lambda}\in\mathcal{N}_{0}^{\zeta}$. By Lemma
\ref{lem:Rational4PowersDistance}, for every $\xi\in\mathbb{Z}^{3}$,
\[
\left|\left|\xi\right|^{2}-\lambda\right|<\frac{1}{2}\,\Longrightarrow\left|\left|\xi-\zeta\right|^{2}-\lambda\right|>\frac{1}{2}.
\]
Hence we can write 
\begin{align*}
\sum\nolimits ^{2}\frac{1}{\left|\left|\xi-\zeta\right|^{2}-\lambda\right|\left|\left|\xi\right|^{2}-\lambda\right|}&=\sum\nolimits ^{3}\frac{1}{\left|\left|\xi-\zeta\right|^{2}-\lambda\right|\left|\left|\xi\right|^{2}-\lambda\right|}\\
&+\sum\nolimits ^{4}\frac{1}{\left|\left|\xi-\zeta\right|^{2}-\lambda\right|\left|\left|\xi\right|^{2}-\lambda\right|}\\
&+\sum\nolimits ^{5}\frac{1}{\left|\left|\xi-\zeta\right|^{2}-\lambda\right|\left|\left|\xi\right|^{2}-\lambda\right|}
\end{align*}
where in $\sum^{3}$ the summation is over $\xi\in\mathbb{Z}^{3}$
such that: 
\[
\frac{1}{2}\leq\left|\left|\xi\right|^{2}-\lambda\right|<L,\,\left|\left\langle \xi,\zeta\right\rangle \right|<L,\,\left|\left|\xi-\zeta\right|^{2}-\lambda\right|\geq\frac{1}{2},
\]
in $\sum^{4}$ the summation is over $\xi\in\mathbb{Z}^{3}$ such
that: 
\[
\left|\left|\xi\right|^{2}-\lambda\right|<\frac{1}{2},\,\left|\left\langle \xi,\zeta\right\rangle \right|<L,\,\left|\left|\xi-\zeta\right|^{2}-\lambda\right|>\frac{1}{2},
\]
and in $\sum^{5}$ the summation is over $\xi\in\mathbb{Z}^{3}$ such
that: 
\[
\frac{1}{2}\leq\left|\left|\xi\right|^{2}-\lambda\right|<L,\,\left|\left\langle \xi,\zeta\right\rangle \right|<L,\,\left|\left|\xi-\zeta\right|^{2}-\lambda\right|<\frac{1}{2}.
\]
Using Lemma \ref{lem:IntSphericalStripsEst} we have
\[
\sum\nolimits ^{3}\frac{1}{\left|\left|\xi-\zeta\right|^{2}-\lambda\right|\left|\left|\xi\right|^{2}-\lambda\right|}\ll\sum\nolimits ^{3}1\ll\sum_{\left|n-\lambda\right|<L}Ln^{\varepsilon}\ll L^{2}\lambda^{\varepsilon}=\lambda^{2\delta+\varepsilon}.
\]
For the second sum, we get by Cauchy-Schwarz and Lemma \ref{lem:IntSphericalStripsEst}
\begin{align*}
\sum\nolimits ^{4}\frac{1}{\left|\left|\xi-\zeta\right|^{2}-\lambda\right|\left|\left|\xi\right|^{2}-\lambda\right|}&\ll\sum\nolimits ^{4}\frac{1}{\left|\left|\xi\right|^{2}-\lambda\right|}\\
&\ll\left(\sum\nolimits ^{4}\frac{1}{\left|\left|\xi\right|^{2}-\lambda\right|^{2}}\right)^{1/2}\left(\sum\nolimits ^{4}1\right)^{1/2}\\
&\ll\left\|G_{\lambda,L}\right\|_{2}\left(\sum\nolimits ^{4}1\right)^{1/2}\\
&\ll\left\|G_{\lambda,L}\right\|_{2}\left(Ln_{\lambda}^{\varepsilon}\right)^{1/2}\\
&\ll\left\|G_{\lambda,L}\right\|_{2}L^{1/2}\lambda^{\varepsilon/2}\\
&=\left\|G_{\lambda,L}\right\|_{2}\lambda^{\delta/2+\varepsilon/2}.
\end{align*}
For the third sum we note that for $\left|\left|\xi\right|^{2}-\lambda\right|<L,\,\left|\left\langle \xi,\zeta\right\rangle \right|<L$
(for every large enough $L$) we have:
\begin{align*}
\left|\left|\xi-\zeta\right|^{2}-\lambda\right|&=\left|\left|\xi\right|^{2}-\lambda-2\left\langle \xi,\zeta\right\rangle +\left|\zeta\right|^{2}\right|\\
&\leq\left|\left|\xi\right|^{2}-\lambda\right|+2\left|\left\langle \xi,\zeta\right\rangle \right|+\left|\zeta\right|^{2}\\
&<3L+\left|\zeta\right|^{2}\leq4L
\end{align*}
and
\[
\left|\left\langle \xi-\zeta,\zeta\right\rangle \right|=\left|\left\langle \xi,\zeta\right\rangle -\left|\zeta\right|^{2}\right|\leq\left|\left\langle \xi,\zeta\right\rangle \right|+\left|\zeta\right|^{2}<2L,
\]
so 
\[
\sum\nolimits ^{2}\frac{1}{\left|\left|\xi-\zeta\right|^{2}-\lambda\right|^{2}}\leq\sum\nolimits ^{6}\frac{1}{\left|\left|\eta\right|^{2}-\lambda\right|^{2}}
\]
where in $\sum^{6}$ the summation is over $\eta\in\mathbb{Z}^{3}$
such that 
\[
\left|\left|\eta\right|^{2}-\lambda\right|<4L,\,\left|\left\langle \eta,\zeta\right\rangle \right|<2L.
\]
We get that
\begin{align*}
\sum\nolimits ^{5}\frac{1}{\left|\left|\xi-\zeta\right|^{2}-\lambda\right|\left|\left|\xi\right|^{2}-\lambda\right|}&\ll\sum\nolimits ^{5}\frac{1}{\left|\left|\xi-\zeta\right|^{2}-\lambda\right|}\\
&\ll\left(\sum\nolimits ^{5}\frac{1}{\left|\left|\xi-\zeta\right|^{2}-\lambda\right|^{2}}\right)^{1/2}\left(\sum\nolimits ^{5}1\right)^{1/2}\\
&\ll\left(\sum\nolimits ^{2}\frac{1}{\left|\left|\xi-\zeta\right|^{2}-\lambda\right|^{2}}\right)^{1/2}\left(\sum\nolimits ^{5}1\right)^{1/2}\\
&\ll\left(\sum\nolimits ^{6}\frac{1}{\left|\left|\eta\right|^{2}-\lambda\right|^{2}}\right)^{1/2}\left(\sum\nolimits ^{5}1\right)^{1/2}\\
&\ll\left\|G_{\lambda}\right\|_{2}\left(\sum\nolimits ^{5}1\right)^{1/2}
\end{align*}
but since $\left|\left\langle \xi,\zeta\right\rangle \right|<L$ implies
that $\left|\left\langle \xi-\zeta,\zeta\right\rangle \right|<2L$,
and since $\left|\left|\xi-\zeta\right|^{2}-\lambda\right|<1/2$ implies
that $\left|\xi-\zeta\right|^{2}=n_{\lambda}$, Lemma \ref{lem:IntSphericalStripsEst}
yields 
\[
\sum\nolimits ^{5}1\ll\sum_{\substack{\left|\eta\right|^{2}=n_{\lambda}\\
\left|\left\langle \eta,\zeta\right\rangle \right|<2L
}
}1\ll Ln_{\lambda}^{\varepsilon}\ll L\lambda^{\varepsilon}=\lambda^{\delta+\varepsilon},
\]
so
\[
\sum\nolimits ^{5}\frac{1}{\left|\left|\xi-\zeta\right|^{2}-\lambda\right|\left|\left|\xi\right|^{2}-\lambda\right|}\ll\left\|G_{\lambda}\right\|_{2}\lambda^{\delta/2+\varepsilon/2}.
\]
We conclude that
\begin{align*}
\sum\nolimits ^{2}\frac{1}{\left|\left|\xi-\zeta\right|^{2}-\lambda\right|\left|\left|\xi\right|^{2}-\lambda\right|}&\ll\lambda^{2\delta+\varepsilon}+\left\|G_{\lambda,L}\right\|_{2}\lambda^{\delta/2+\varepsilon/2}\\
&+\left\|G_{\lambda}\right\|_{2}\lambda^{\delta/2+\varepsilon/2}\\
&\asymp\lambda^{2\delta+\varepsilon}+2\left\|G_{\lambda}\right\|_{2}\lambda^{\delta/2+\varepsilon/2}.
\end{align*}
Now, assume that $n_{\lambda}\in\mathcal{N}_{1}^{\zeta}$. Cauchy-Schwarz
yields

\begin{align*}
&\sum\nolimits ^{2}\frac{1}{\left|\left|\xi-\zeta\right|^{2}-\lambda\right|\left|\left|\xi\right|^{2}-\lambda\right|}\\
&\ll\left(\sum\nolimits ^{2}\frac{1}{\left|\left|\xi\right|^{2}-\lambda\right|^{2}}\right)^{1/2}\left(\sum\nolimits ^{2}\frac{1}{\left|\left|\xi-\zeta\right|^{2}-\lambda\right|^{2}}\right)^{1/2}\\
&\ll\left\|G_{\lambda,L}\right\|_{2}\left(\sum\nolimits ^{2}\frac{1}{\left|\left|\xi-\zeta\right|^{2}-\lambda\right|^{2}}\right)^{1/2}\\
&\ll\left\|G_{\lambda,L}\right\|_{2}\left(\sum\nolimits ^{6}\frac{1}{\left|\left|\eta\right|^{2}-\lambda\right|^{2}}\right)^{1/2}\\
&=\left\|G_{\lambda,L}\right\|_{2}\left(\sum\nolimits ^{7}\frac{1}{\left|\left|\eta\right|^{2}-\lambda\right|^{2}}+\sum\nolimits ^{8}\frac{1}{\left|\left|\eta\right|^{2}-\lambda\right|^{2}}\right)^{1/2}
\end{align*}
where in $\sum^{7}$ the summation is over $\eta\in\mathbb{Z}^{3}$
such that: 
\[
\left|\left|\eta\right|^{2}-\lambda\right|<4L,\,\left|\left\langle \eta,\zeta\right\rangle \right|<2L,\,\left|\eta\right|^{2}\neq n_{\lambda},
\]
and in $\sum^{8}$ the summation is over $\eta\in\mathbb{Z}^{3}$
such that: 
\[
\left|\left|\eta\right|^{2}-\lambda\right|<4L,\,\left|\left\langle \eta,\zeta\right\rangle \right|<2L,\,\left|\eta\right|^{2}=n_{\lambda}.
\]
For $\left|\eta\right|^{2}\neq n_{\lambda},$ we clearly have $\left|\left|\eta\right|^{2}-\lambda\right|\geq\frac{1}{2}$,
so using Lemma \ref{lem:IntSphericalStripsEst}, we get
\[
\sum\nolimits ^{7}\frac{1}{\left|\left|\eta\right|^{2}-\lambda\right|^{2}}\ll\sum\nolimits ^{7}1\ll\sum\limits _{\left|n-\lambda\right|<4L}Ln^{\varepsilon}\ll\lambda^{2\delta+\varepsilon}.
\]
For the last sum, we use Lemma \ref{lem:IntSphericalStripsEst} again
to get
\[
\frac{\sum^{8}\frac{1}{\left|\left|\eta\right|^{2}-\lambda\right|^{2}}}{\left\|G_{\lambda}\right\|_{2}^{2}}\ll\frac{\frac{Ln_{\lambda}^{\varepsilon}}{\left(n_{\lambda}-\lambda\right)^{2}}}{\left\|G_{\lambda}\right\|_{2}^{2}}\ll\frac{\frac{L\lambda^{\varepsilon}}{\left(n_{\lambda}-\lambda\right)^{2}}}{\left\|G_{\lambda}\right\|_{2}^{2}}\asymp\frac{\frac{L\lambda^{\varepsilon}}{\left(n_{\lambda}-\lambda\right)^{2}}}{\sum\limits _{n=0}^{\infty}\frac{r_{3}\left(n\right)}{\left(n-\lambda\right)^{2}}};
\]
$\forall n_{0}$ we have $\sum\limits _{n=0}^{\infty}\frac{r_{3}\left(n\right)}{\left(n-\lambda\right)^{2}}\geq\frac{r_{3}\left(n_{0}\right)}{\left(n_{0}-\lambda\right)^{2}}$,
so
\[
\frac{\frac{L\lambda^{\varepsilon}}{\left(n_{\lambda}-\lambda\right)^{2}}}{\sum\limits _{n=0}^{\infty}\frac{r_{3}\left(n\right)}{\left(n-\lambda\right)^{2}}}\leq\frac{L\lambda^{\varepsilon}}{r_{3}\left(n_{\lambda}\right)}.
\]
If we write $n_{\lambda}=4^{a}n_{1},$ with $4\nmid n_{1}$, then
since $n_{\lambda}\in\mathcal{N}_{1}^{\zeta}$ we know that $n_{\lambda}\leq4^{a_{\zeta}}n_{1}$,
so 
\begin{align*}
\frac{L\lambda^{\varepsilon}}{r_{3}\left(n_{\lambda}\right)}=\frac{L\lambda^{\varepsilon}}{r_{3}\left(n_{1}\right)}&\ll\frac{L\lambda^{\varepsilon}}{n_{1}^{1/2-\varepsilon}}\\
&\ll\frac{L\lambda^{\varepsilon}}{n_{\lambda}^{1/2-\varepsilon}}
\\&\ll\frac{L\lambda^{\varepsilon}}{\lambda^{1/2-\varepsilon}}\\
&=L\lambda^{-1/2+2\varepsilon}=\lambda^{\delta-1/2+2\varepsilon}
\end{align*}
and therefore
\[
\sum\nolimits ^{8}\frac{1}{\left|\left|\eta\right|^{2}-\lambda\right|^{2}}\ll\left\|G_{\lambda}\right\|_{2}^{2}\lambda^{\delta-1/2+2\varepsilon},
\]
so now we conclude that
\begin{align*}
\sum\nolimits ^{2}\frac{1}{\left|\left|\xi-\zeta\right|^{2}-\lambda\right|\left|\left|\xi\right|^{2}-\lambda\right|}&\ll\left\|G_{\lambda,L}\right\|_{2}\left(\lambda^{2\delta+\varepsilon}+\left\|G_{\lambda}\right\|_{2}^{2}\lambda^{\delta-1/2+2\varepsilon}\right)^{1/2}\\
&\leq\left\|G_{\lambda,L}\right\|_{2}\left(\lambda^{\delta+\varepsilon/2}+\left\|G_{\lambda}\right\|_{2}\lambda^{\delta/2-1/4+\varepsilon}\right)\\
&\asymp\left\|G_{\lambda}\right\|_{2}\left(\lambda^{\delta+\varepsilon/2}+\left\|G_{\lambda}\right\|_{2}\lambda^{\delta/2-1/4+\varepsilon}\right).
\end{align*}
Either way we got that
\begin{align*}
\sum\nolimits ^{2}\frac{1}{\left|\left|\xi-\zeta\right|^{2}-\lambda\right|\left|\left|\xi\right|^{2}-\lambda\right|}&\ll\lambda^{2\delta+\varepsilon}+2\left\|G_{\lambda}\right\|\lambda^{\delta/2+\varepsilon/2}\\
&+\left\|G_{\lambda}\right\|_{2}\lambda^{\delta+\varepsilon/2}+\left\|G_{\lambda}\right\|_{2}^{2}\lambda^{\delta/2-1/4+\varepsilon},
\end{align*}
so we have
\begin{align*}
\left|\left\langle e_{\zeta}g_{\lambda,L},g_{\lambda,L}\right\rangle \right|&=\frac{\left|\left\langle e_{\zeta}G_{\lambda,L},G_{\lambda,L}\right\rangle \right|}{\left\|G_{\lambda,L}\right\|_{2}^{2}}\\
&\ll\frac{\left|\left\langle e_{\zeta}G_{\lambda,L},G_{\lambda,L}\right\rangle \right|}{\left\|G_{\lambda}\right\|_{2}^{2}}\\
&\ll\frac{\lambda^{2\delta+\varepsilon}}{\left\|G_{\lambda}\right\|_{2}^{2}}+\frac{\lambda^{-\delta/2+1/4+\varepsilon/2}+2\lambda^{\delta/2+\varepsilon/2}+\lambda^{\delta+\varepsilon/2}}{\left\|G_{\lambda}\right\|_{2}}\\
&+\lambda^{\delta/2-1/4+\varepsilon}\\
&\ll\frac{\lambda^{2\delta+\varepsilon}}{\lambda^{1/2-\varepsilon}}+\frac{\lambda^{-\delta/2+1/4+\varepsilon/2}+2\lambda^{\delta/2+\varepsilon/2}+\lambda^{\delta+\varepsilon/2}}{\lambda^{1/4-\varepsilon/2}}\\
&+\lambda^{\delta/2-1/4+\varepsilon}\\
&=\lambda^{2\delta-1/2+2\varepsilon}+\lambda^{-\delta/2+\varepsilon}+3\lambda^{\delta/2-1/4+\varepsilon}+\lambda^{\delta-1/4+\varepsilon},
\end{align*}
and since $0<\delta<1/4$, the proposition follows.
\end{proof}
Theorem \ref{thm:Main1} now easily follows by the density of trigonometric
polynomials in $C^{\infty}\left(\mathbb{T}^{3}\right)$ in the uniform
norm:
\begin{thm*}
For every $a\in C^{\infty}\left(\mathbb{T}^{3}\right)$, we have 
\[
\int_{\mathbb{T}^{3}}a\left(x\right)\left|g_{\lambda}\left(x\right)\right|^{2}\mbox{d}x\to\frac{1}{\mbox{area}\left(\mathbb{T}^{3}\right)}\int_{\mathbb{T}^{3}}a\left(x\right)\mbox{d}x
\]
 as $\lambda\to\infty$ along $\Lambda$.\end{thm*}
\begin{proof}
Let $P\left(x\right)=\sum\limits _{\left|\zeta\right|\leq J}p_{\zeta}e_{\zeta}\left(x\right)$
be a trigonometric polynomial. From Proposition \ref{prop:MainProp}
we have
\[
\left\langle Pg_{\lambda},g_{\lambda}\right\rangle =\sum_{\left|\zeta\right|\leq J}p_{\zeta}\left\langle e_{\zeta}g_{\lambda},g_{\lambda}\right\rangle \to p_{\left(0,0,0\right)}=\frac{1}{\mbox{area}\left(\mathbb{T}^{3}\right)}\int_{\mathbb{T}^{3}}P\left(x\right)\mbox{d}x
\]
as $\lambda\to\infty$.

Let $\varepsilon>0$. For every $a\in C^{\infty}\left(\mathbb{T}^{3}\right)$,
there exists a trigonometric polynomial $P$ such that $\left\|a-P\right\|_{\infty}<\varepsilon$.
Thus for every large enough $\lambda\in\Lambda$
\begin{align*}
&\left|\left\langle ag_{\lambda},g_{\lambda}\right\rangle -\frac{1}{\mbox{area}\left(\mathbb{T}^{3}\right)}\int_{\mathbb{T}^{3}}a\left(x\right)\mbox{d}x\right|\\
&\leq\left|\left\langle ag_{\lambda},g_{\lambda}\right\rangle -\left\langle Pg_{\lambda},g_{\lambda}\right\rangle \right|+\left|\left\langle Pg_{\lambda},g_{\lambda}\right\rangle -\frac{1}{\mbox{area}\left(\mathbb{T}^{3}\right)}\int_{\mathbb{T}^{3}}P\left(x\right)\mbox{d}x\right|\\
&+\left|\frac{1}{\mbox{area}\left(\mathbb{T}^{3}\right)}\int_{\mathbb{T}^{3}}P\left(x\right)\mbox{d}x-\frac{1}{\mbox{area}\left(\mathbb{T}^{3}\right)}\int_{\mathbb{T}^{3}}a\left(x\right)\mbox{d}x\right|\\
&<2\left\|a-P\right\|_{\infty}+\varepsilon<3\varepsilon\qedhere.
\end{align*}

\end{proof}

\section{The Irrational Torus}

\subsection{Basic Setup}

Let $\mathbb{T}^{3}=\mathbb{R}^{3}/2\pi\mathcal{L}_{0}$ be a flat
three-dimensional torus, where 
\[
\mathcal{L}_{0}=\mathbb{Z}\left(a,0,0\right)\oplus\mathbb{Z}\left(0,b,0\right)\oplus\mathbb{Z}\left(0,0,c\right)
\]
is a lattice, such that $1/a^{2},1/b^{2},1/c^{2}\in\mathbb{R}$ are
independent over $\mathbb{Q}$. We also demand that at least one of
the ratios $b^{2}/a^{2},\, c^{2}/a^{2},\, c^{2}/b^{2}$ will be an
irrational of finite type $\tau$, as in Definition \ref{def:Type1Irrational}
(without loss of generality assume it to be $c^{2}/a^{2}$). 

The norm of a lattice vector $\xi=\left(\xi_{1}/a,\xi_{2}/b,\xi_{3}/c\right)\in\mathcal{L}$
is 
\[
\xi_{1}^{2}/a^{2}+\xi_{2}^{2}/b^{2}+\xi_{3}^{2}/c^{2}
\]
so if $\eta=\left(\eta_{1}/a,\eta_{2}/b,\eta_{3}/c\right)$ is another
vector of $\mathcal{L}$ of the same norm, we have
\[
\left(\xi_{1}^{2}-\eta_{1}^{2}\right)/a^{2}+\left(\xi_{2}^{2}-\eta_{2}^{2}\right)/b^{2}+\left(\xi_{3}^{2}-\eta_{3}^{2}\right)/c^{2}=0
\]
and since $1/a^{2},1/b^{2},1/c^{2}$ are independent over the rationals
we get that $\eta_{i}=\pm\xi_{i}$ for $1\leq i\leq3$.

We conclude that for $n\in\mathcal{N}$ we have $r_{\mathcal{L}}\left(n\right)=1,2,4$
or $8$.

Weyl's law for the torus, establishing the asymptotics of the counting
function $N\left(x\right)$ of eigenvalues below $x$, is equivalent
to counting the number of points of the standard lattice $\mathbb{Z}^{3}$
in an ellipsoid:
\begin{align*}
N\left(x\right)&=\#\left\{ \left(\xi_{1},\xi_{2},\xi_{3}\right)\in\mathbb{Z}^{3}:\,\xi_{1}^{2}/a^{2}+\xi_{2}^{2}/b^{2}+\xi_{3}^{2}/c^{2}\leq x\right\}\\
&=\frac{4}{3}\pi abcx^{3/2}+O\left(x^{\theta}\right).
\end{align*}
The trivial bound on the remainder term is $\theta=1$. We will need
a bound $\theta<1$, such as the bound due to Hlawka \cite{Hlawka}
using Poisson summation which translates to $\theta=3/4.$

Note that since the multiplicities $r_{\mathcal{L}}\left(n\right)$
are bounded, we have
\[
\#\left\{ \lambda\in\Lambda:\,\lambda\leq X\right\} \asymp N\left(X\right)\asymp X^{3/2}.
\]
We will need to analyze the spacing between the elements of $\Lambda$.
We first notice that for most of the elements, the nearest neighbor
cannot be too far:

For $\varepsilon>0,$ define

\[
\Lambda_{1}=\Lambda_{1}^{\varepsilon}=\left\{ \lambda\in\Lambda:\,\lambda\geq1,\,\left(\lambda,\lambda+\lambda^{-1/2+\varepsilon}\right)\cap\Lambda\neq\emptyset\right\}.
\]
We claim that this is a density one set in $\Lambda.$ To show this,
define
\[
B_{1}=\Lambda\setminus\Lambda_{1}=\left\{ \lambda\in\Lambda:\,\lambda<1\right\} \cup\left\{ \lambda\in\Lambda:\,\lambda\geq1,\,\left(\lambda,\lambda+\lambda^{-1/2+\varepsilon}\right)\cap\Lambda=\emptyset\right\}.
\]
We want to show that $B_{1}$ is a density zero set in $\Lambda$,
that is, 
\[
\#\left\{ \lambda\in B_{1}:\,\lambda\leq X\right\} =o\left(\#\left\{ \lambda\in\Lambda:\,\lambda\leq X\right\} \right)=o\left(X^{3/2}\right).
\]
We have $\#\left\{ \lambda\in\Lambda:\,\lambda<1\right\} =O\left(1\right)$,
so we only need to check that
\[
\#\left\{ \lambda\in\tilde{B_{1}}:\,\lambda\leq X\right\} \leq X^{3/2-\varepsilon}
\]
where
\[
\tilde{B}_{1}=\left\{ \lambda\in\Lambda:\,\lambda\geq1,\,\left(\lambda,\lambda+\lambda^{-1/2+\varepsilon}\right)\cap\Lambda=\emptyset\right\} .
\]
Indeed, the intervals $\left(\lambda,\lambda+\lambda^{-1/2+\varepsilon}\right),\,\lambda\in\tilde{B_{1}}$
are disjoint, and therefore 
\begin{align*}
\#\left\{ \lambda\in\tilde{B_{1}}:\,\lambda\leq X\right\} \cdot X^{-1/2+\varepsilon}&\leq\mbox{meas}\left(\bigcup_{\begin{subarray}{c}
\lambda\in\tilde{B_{1}}\\
\lambda\leq X
\end{subarray}}\left(\lambda,\lambda+\lambda^{-1/2+\varepsilon}\right)\right)\\
&\leq\mbox{meas}\left(\left(1,X+1\right)\right)=X
\end{align*}
so $\#\left\{ \lambda\in\tilde{B_{1}}:\,\lambda\leq X\right\} \leq X^{3/2-\varepsilon}$.

\subsection{Lattice Points in Thin Spherical Shells}

For $0<L\leq1$, define $A\left(\lambda,L\right)$ to be the set of
lattice points in the spherical shell $\lambda-L<\left|x\right|^{2}<\lambda+L$:
\[
A\left(\lambda,L\right)=\left\{ \xi\in\mathcal{L}:\,\lambda-L<\left|\xi\right|^{2}<\lambda+L\right\}.
\]
Define
\[
\tilde{A}\left(\lambda,L\right)=\left(\lambda-L,\lambda+L\right)\cap\Lambda.
\]
We want to show that for $L=\lambda^{-\delta}$, $0<\delta<\min\left\{ \left(1-\theta\right)/2-\varepsilon,1/\tau-\varepsilon\right\} $
(for $\varepsilon>0$ small enough there is such $\delta$, since
$\theta<1)$, we have a density one set in $\Lambda$ such that $\#\tilde{A}\left(\lambda,3L\right)\leq L\lambda^{1/2+2\varepsilon}$
for every element $ \lambda $ of this set. In order to show this, we need the
following lemma:
\begin{lem}
\label{lem:IrrationalThinShells}Let $L=\lambda^{-\delta}$, $0<\delta<\min\left\{ \left(1-\theta\right)/2-\varepsilon,1/\tau-\varepsilon\right\} $.
We have 
\[
\sum\limits _{\begin{subarray}{c}
\lambda\in\Lambda_{1}\\
X<\lambda\leq2X
\end{subarray}}\#\tilde{A}\left(\lambda,3L\right)\ll X^{2-\delta+\varepsilon}.
\]
\end{lem}
\begin{proof}
For every $\lambda\in\Lambda_{1}$, $X<\lambda\leq2X$, choose $\xi\in\mathcal{L}$
such that $\left|\xi\right|^{2}$ is the smallest element in $\mathcal{N}$
greater than $\lambda$. Since $\lambda\in\Lambda_{1}$ and $\delta<\left(1-\theta\right)/2-\varepsilon<1/2-\varepsilon$
we know that $\xi\in A\left(\lambda,X^{-\delta}\right)$, and we conclude
that
\[
\#\tilde{A}\left(\lambda,3X^{-\delta}\right)\ll\#A\left(\lambda,3X^{-\delta}\right)\leq\#A\left(\left|\xi\right|^{2},4X^{-\delta}\right)=\sum_{\begin{subarray}{c}
\eta\in\mathcal{L}\\
\left|\left|\eta\right|^{2}-\left|\xi\right|^{2}\right|<4X^{-\delta}
\end{subarray}}1
\]
so
\begin{align}
\sum_{\begin{subarray}{c}
\lambda\in\Lambda_{1}\\
X<\lambda\leq2X
\end{subarray}}\#\tilde{A}\left(\lambda,3L\right)&\leq\sum_{\begin{subarray}{c}
\lambda\in\Lambda_{1}\\
X<\lambda\leq2X
\end{subarray}}\#\tilde{A}\left(\lambda,3X^{-\delta}\right)\nonumber\\
&\ll\sum_{\begin{subarray}{c}
\xi,\eta\in\mathcal{L}\\
\left|\xi\right|^{2}\leq3X\\
\left|\left|\xi\right|^{2}-\left|\eta\right|^{2}\right|<4X^{-\delta}
\end{subarray}}1\label{eq:IrrationalLatticePointsNo}\\
&\leq\#\left\{ \xi,\eta\in\mathcal{L}:\,\left|\xi\right|^{2},\left|\eta\right|^{2}\leq4X,\,\left|\left|\xi\right|^{2}-\left|\eta\right|^{2}\right|<4X^{-\delta}\right\} .\nonumber 
\end{align}
Thus we want to bound the number of solutions of the RHS of \eqref{eq:IrrationalLatticePointsNo},
which is a quadratic Diophantine inequality with a shrinking target,
by $O\left(X^{2-\delta+\varepsilon}\right)$.

We transform the problem into linear Diophantine inequality as follows:
Write $\xi=\left(\xi_{1}/a,\xi_{2}/b,\xi_{3}/c\right)$, $\eta=\left(\eta_{1}/a,\eta_{2}/b,\eta_{3}/c\right)$,
then
\[
\left|\xi\right|^{2}-\left|\eta\right|^{2}=\frac{1}{a^{2}}\left(\xi_{1}^{2}-\eta_{1}^{2}\right)+\frac{1}{b^{2}}\left(\xi_{2}^{2}-\eta_{2}^{2}\right)+\frac{1}{c^{2}}\left(\xi_{3}^{2}-\eta_{3}^{2}\right)=\frac{z_{1}}{a^{2}}+\frac{z_{2}}{b^{2}}+\frac{z_{3}}{c^{2}}
\]
where $z_{j}=\xi_{j}^{2}-\eta_{j}^{2}$ is of size $\left|z_{j}\right|\leq CX$
for some constant $C$.

Assume first that for every $1\leq j\leq3$: $\xi_{j}^{2}\neq\eta_{j}^{2}$,
so $z_{j}\neq0$. The number of solutions to
\[
z_{j}=\xi_{j}^{2}-\eta_{j}^{2}=\left(\xi_{j}-\eta_{j}\right)\left(\xi_{j}+\eta_{j}\right)
\]
is bounded by the number of divisors of $z_{j}$ which is $O\left(z_{j}^{\varepsilon}\right)=O\left(X^{\varepsilon}\right).$
Thus we have
\begin{align*}
&\#\left\{ \xi,\eta\in\mathcal{L}:\,\left|\xi\right|^{2},\left|\eta\right|^{2}\leq4X,\,\xi_{j}^{2}\neq\eta_{j}^{2},\,\left|\left|\xi\right|^{2}-\left|\eta\right|^{2}\right|<4X^{-\delta}\right\}\\
&\ll X^{\varepsilon}\cdot\#\left\{ \left(z_{1},z_{2},z_{3}\right)\in\mathbb{Z}^{3}:\,1\leq\left|z_{j}\right|\leq CX,\,\left|\frac{z_{1}}{a^{2}}+\frac{z_{2}}{b^{2}}+\frac{z_{3}}{c^{2}}\right|<4X^{-\delta}\right\} .
\end{align*}
Denote 
\[
A_{X}=\left\{ \left(z_{1},z_{2},z_{3}\right)\in\mathbb{Z}^{3}:\,1\leq\left|z_{j}\right|\leq CX,\,\left|\frac{z_{1}}{a^{2}}+\frac{z_{2}}{b^{2}}+\frac{z_{3}}{c^{2}}\right|<4X^{-\delta}\right\} .
\]
We will show that $\#A_{X}\ll X^{2-\delta}$: first note that for
every $\left(z_{1},z_{2},z_{3}\right)\in A_{X}$, assuming that $X$
is large enough we have
\[
\left|\frac{c^{2}}{a^{2}}z_{1}+\frac{c^{2}}{b^{2}}z_{2}+z_{3}\right|<\frac{1}{4}
\]
so $z_{3}$ is uniquely determined by the values of $z_{1},z_{2}$,
and we have 
\[
\left\|\frac{c^{2}}{a^{2}}z_{1}+\frac{c^{2}}{b^{2}}z_{2}\right\|<4X^{-\delta}
\]
so 
\begin{align*}
&\#A_{X}\leq\#\left\{ \left(z_{1},z_{2}\right)\in\mathbb{Z}^{2}:\,1\leq\left|z_{j}\right|\leq CX,\,\left\|\frac{c^{2}}{a^{2}}z_{1}+\frac{c^{2}}{b^{2}}z_{2}\right\|<4X^{-\delta}\right\} \\
&=\sum_{\begin{subarray}{c}
1\leq\left|z_{2}\right|\leq CX\\
0\leq j\leq1
\end{subarray}}\#\left\{ z_{1}\in\mathbb{N}:\, z_{1}\leq\lfloor CX\rfloor,\,\left\|\left(-1\right)^{j}\frac{c^{2}}{a^{2}}z_{1}+\frac{c^{2}}{b^{2}}z_{2}\right\|<4X^{-\delta}\right\}\ll\\
&\sum_{\begin{subarray}{c}
1\leq\left|z_{2}\right|\leq CX\\
0\leq j_{1},j_{2}\leq1
\end{subarray}}\#\left\{ z_{1}\in\mathbb{N}:\, z_{1}\leq\lfloor CX\rfloor,\,\left\{ \left(-1\right)^{j_{1}}\frac{c^{2}}{a^{2}}z_{1}+\left(-1\right)^{j_{2}}\frac{c^{2}}{b^{2}}z_{2}\right\} <4X^{-\delta}\right\}.
\end{align*}

In the Appendix \eqref{sub:AppDiscrepancy}, we show that for every sequence $ x_n=\alpha n+\beta $, where $ \alpha\in \mathbb{R} $ is an irrational of finite type $ \tau $ and $ \beta\in \mathbb{R} $, we have an upper bound for the discrepancy $ D_N $ of the sequence $ \left(x_n\right) $: \[ D_N\le cN^{-1/\tau+\varepsilon} \] where $ c=c\left(\alpha,\varepsilon\right) $ is a constant which does not depend on $ \beta $ (formula \eqref{eq:DiscrepancyEst}). 

For every $z_{2}$, using this bound for the sequence $ x_n=\alpha n+\beta $, with $ \alpha=\pm c^{2}/a^{2} $ (which is an irrational of finite type $ \tau $) and $\beta=\pm \left(c^{2}/b^{2}\right)z_{2}$, we get that (since  $\delta<1/\tau-\varepsilon$)
\[
\#\left\{ z_{1}\in\mathbb{N}:\, z_{1}\leq\lfloor CX\rfloor,\,\left\{ \pm\frac{c^{2}}{a^{2}}z_{1}\pm\frac{c^{2}}{b^{2}}z_{2}\right\} <4X^{-\delta}\right\} \ll X^{1-\delta}
\]
so we conclude that
\[
\#A_{X}\ll X^{2-\delta}.
\]
Assume now that exactly one of the $z_{j}$ equals zero: without loss
of generality assume that $z_{1}=0$ and $z_{2},z_{3}\neq0$. Since
for $j=2,3$ the number of solutions to $z_{j}=\xi_{j}^{2}-\eta_{j}^{2}$
is $O\left(X^{\varepsilon}\right)$, and the number of solutions to
$\xi_{1}^{2}=\eta_{1}^{2}$ is $O\left(X^{1/2}\right)$ we conclude
that the number of solutions of the RHS of \eqref{eq:IrrationalLatticePointsNo}
(under our assumption) is bounded by
\[
X^{1/2+\varepsilon}\cdot\#\left\{ \left(z_{2},z_{3}\right)\in\mathbb{Z}^{2}:\,1\leq\left|z_{j}\right|\leq CX,\,\left|\frac{z_{2}}{b^{2}}+\frac{z_{3}}{c^{2}}\right|<4X^{-\delta}\right\}.
\]
Denote 
\[
B_{X}=\left\{ \left(z_{2},z_{3}\right)\in\mathbb{Z}^{2}:\,1\leq\left|z_{j}\right|\leq CX,\,\left|\frac{z_{2}}{b^{2}}+\frac{z_{3}}{c^{2}}\right|<4X^{-\delta}\right\} .
\]
Note that for every $\left(z_{2},z_{3}\right)\in B_{X}$, assuming
that $X$ is large enough we have
\[
\left|\frac{c^{2}}{b^{2}}z_{2}+z_{3}\right|<\frac{1}{4}
\]
so $z_{3}$ is uniquely determined by the value of $z_{2}$, and therefore
$\#B_{X}\ll X$. 

If exactly two of the $z_{j}$ equal zero, for instance $z_{1}=z_{2}=0$,
then for large enough $X$ we must have $z_{3}=0$, so the number
of solutions of the RHS of \eqref{eq:IrrationalLatticePointsNo} (under
our assumption) is $O$$\left(X^{3/2}\right)$. Since $\delta<\left(1-\theta\right)/2-\varepsilon<1/2$
the lemma is proved. 
\end{proof}
Define
\[
\Lambda_{2}=\Lambda_{2}^{\varepsilon,\delta}=\left\{ \lambda\in\Lambda_{1}:\,\#\tilde{A}\left(\lambda,3L\right)\leq L\lambda^{1/2+2\varepsilon}\right\}.
\]
We will show that this is a density one set in $\Lambda_{1}$ (and
hence in $\Lambda$):

Define
\[
B_{2}=\Lambda_{1}\setminus\Lambda_{2}=\left\{ \lambda\in\Lambda_{1}:\,\#\tilde{A}\left(\lambda,3L\right)>L\lambda^{1/2+2\varepsilon}\right\}.
\]
We will check that 
\[
\#\left\{ \lambda\in B_{2}:\,\lambda\leq X\right\} \ll X^{3/2-\varepsilon}.
\]
Indeed, from Lemma \ref{lem:IrrationalThinShells} we have
\begin{align*}
\#\left\{ \lambda\in B_{2}:\, X<\lambda\leq2X\right\} \cdot X^{1/2+2\varepsilon-\delta}&<\sum_{\begin{subarray}{c}
\lambda\in B_{2}\\
X<\lambda\leq2X
\end{subarray}}\#\tilde{A}\left(\lambda,3L\right)\\
&\leq\sum_{\begin{subarray}{c}
\lambda\in\Lambda_{1}\\
X<\lambda\leq2X
\end{subarray}}\#\tilde{A}\left(\lambda,3L\right)\\
&\ll X^{2-\delta+\varepsilon}
\end{align*}
so there exist $ C>0 $ and $ M>0 $ such that for all $ X\ge M/2 $
\begin{equation}
\#\left\{ \lambda\in B_{2}:\, X<\lambda\leq2X\right\} \le CX^{3/2-\varepsilon}.\label{eq:MeasureZeroXto2X}
\end{equation}
Note that
\[
\#\left\{ \lambda\in B_{2}:\,\lambda\leq X\right\} =\sum_{k=0}^{\infty}\#\left\{ \lambda\in B_{2}:\, X/2^{k+1}<\lambda\leq X/2^{k}\right\} 
\]
(and actually the summation over $k$ is finite). From \eqref{eq:MeasureZeroXto2X}, for every $k\geq0$ such that $ X/2^{k+1}\ge M/2 $ (so $ k\le\lfloor\log_2\left(X/M\right)\rfloor $), we have 
\[
\#\left\{ \lambda\in B_{2}:\, X/2^{k+1}<\lambda\leq X/2^{k}\right\} \le\ C\left(\frac{X}{2^{k+1}}\right)^{3/2-\varepsilon}
\]
so for $ X\ge M $ we have
\begin{align*}
&\sum_{k=0}^{\infty}\#\left\{ \lambda\in B_{2}:\, X/2^{k+1}<\lambda\leq X/2^{k}\right\}\\
&=\sum_{k=0}^{\lfloor\log_2\left(X/M\right)\rfloor}\#\left\{ \lambda\in B_{2}:\, X/2^{k+1}<\lambda\leq X/2^{k}\right\}\\
&+\sum_{k=\lfloor\log_2\left(X/M\right)\rfloor +1}^{\infty}\#\left\{ \lambda\in B_{2}:\, X/2^{k+1}<\lambda\leq X/2^{k}\right\}\\
&\le CX^{3/2-\varepsilon}\sum_{k=0}^{\lfloor\log_2\left(X/M\right)\rfloor}\frac{1}{2^{\left(k+1\right)\cdot\left(3/2-\varepsilon\right)}}+\#\left\{ \lambda\in B_{2}:\, \lambda\le M\right\}\\
&\ll X^{3/2-\varepsilon}
\end{align*}
as we claimed.

\subsection{Bounds for the Green's Function and Truncation}

We first give a lower bound for the $L^{2}$-norm of the Green's function
$G_{\lambda}$:
\begin{lem}
\label{lem:IrrationalNormBound}For every $\lambda\in\Lambda_{2}$,
we have 
\[
\left\|G_{\lambda}\right\|_{2}^{2}\gg\lambda^{1-2\varepsilon}.
\]
\end{lem}
\begin{proof}
Take $\lambda_{0}\in\left(\lambda,\lambda+\lambda^{-1/2+\varepsilon}\right)\cap\Lambda$.
Let $n_{0}$ be some norm such that $\lambda<n_{0}<\lambda_{0}$.

We have 
\begin{align*}
\left\|G_{\lambda}\right\|_{2}^{2}\asymp\sum_{\xi\in\mathcal{L}}\frac{1}{\left(\left|\xi\right|^{2}-\lambda\right)^{2}}&=\sum_{n\in\mathcal{N}}\frac{r_{\mathcal{L}}\left(n\right)}{\left(n-\lambda\right)^{2}}\\
&\asymp\sum_{n\in\mathcal{N}}\frac{1}{\left(n-\lambda\right)^{2}}\\
&\geq\frac{1}{\left(n_{0}-\lambda\right)^{2}}\\
&>\frac{1}{\left(\lambda_{0}-\lambda\right)^{2}}>\lambda^{1-2\varepsilon}\qedhere.
\end{align*}

\end{proof}
We will now use a truncation procedure.

Recall that for $0<L\leq1$ we defined
\[
A\left(\lambda,L\right)=\left\{ \xi\in\mathcal{L}:\,\lambda-L<\left|\xi\right|^{2}<\lambda+L\right\} 
\]
as the set of lattice points in the spherical shell $\lambda-L<\left|x\right|^{2}<\lambda+L$.
We denote by 
\[
G_{\lambda,L}=-\frac{1}{8\pi^{3}}\sum\limits _{\xi\in A\left(\lambda,L\right)}\frac{\exp\left(i\xi\cdot\left(x-x_{0}\right)\right)}{\left|\xi\right|^{2}-\lambda}
\]
the truncated Green's function, and let $g_{\lambda,L}$ be the $L^{2}$-normalized
truncated Green's function:
\[
g_{\lambda,L}=\frac{G_{\lambda,L}}{\left\|G_{\lambda,L}\right\|_{2}}.
\]

\begin{lem}
\label{lem:Irrational TruncDiff}For $0<\delta<\min\left\{ \left(1-\theta\right)/2-\varepsilon,1/\tau-\varepsilon\right\} ,\, L=\lambda^{-\delta}$,
we have $\left\|g_{\lambda}-g_{\lambda,L}\right\|_{2}\to0$
as $\lambda\to\infty$ along $\Lambda_{2}$.\end{lem}
\begin{proof}
As in \eqref{eq:NormalizedDiffsEst}, we get that
\[
\left\|g_{\lambda}-g_{\lambda,L}\right\|_{2}\leq2\frac{\left\|G_{\lambda}-G_{\lambda,L}\right\|_{2}}{\left\|G_{\lambda}\right\|_{2}}.
\]
We have
\[
\left\|G_{\lambda}-G_{\lambda,L}\right\|^{2}\asymp\sum_{\left|\left|\xi\right|^{2}-\lambda\right|\geq L}\frac{1}{\left(\left|\xi\right|^{2}-\lambda\right)^{2}}.
\]
We recall how to evaluate lattice sums using summation by parts:

Let $0=n_{0}<n_{1}<n_{2}<\dots$ be the set of norms, and
\[
N\left(t\right)=\sum_{n_{k}\leq t}r_{\mathcal{L}}\left(n_{k}\right).
\]
Then for a smooth function $f\left(t\right)$ on $\left[n_{A+1},n_{B}\right]$
we have
\begin{align*}
\sum_{n_{A}<\left|\xi\right|^{2}\leq n_{B}}f\left(\left|\xi\right|^{2}\right)&=N\left(n_{B}\right)f\left(n_{B}\right)\\
&-N\left(n_{A}\right)f\left(n_{A+1}\right)-\int_{n_{A+1}}^{n_{B}}f'\left(t\right)N\left(t\right)\mbox{d}t.
\end{align*}
But since
\[
N\left(x\right)=\frac{4}{3}\pi abcx^{3/2}+O\left(x^{\theta}\right)
\]
and since
\begin{align*}
\frac{4}{3}\pi abcn_{A+1}^{3/2}=N\left(n_{A+1}\right)+O\left(n_{A+1}^{\theta}\right)&=N\left(n_{A}\right)+r_{\mathcal{L}}\left(n_{A+1}\right)+O\left(n_{A+1}^{\theta}\right)\\
&=N\left(n_{A}\right)+O\left(n_{A+1}^{\theta}\right)\\
&=\frac{4}{3}\pi abcn_{A}^{3/2}+O\left(n_{A+1}^{\theta}\right)
\end{align*}
(and therefore $n_{A}\asymp n_{A+1}$, so $\frac{4}{3}\pi abcn_{A+1}^{3/2}=\frac{4}{3}\pi abcn_{A}^{3/2}+O\left(n_{A}^{\theta}\right)$),
we have
\begin{align}
\sum_{n_{A}<\left|\xi\right|^{2}\leq n_{B}}f\left(\left|\xi\right|^{2}\right)&=\left(\frac{4}{3}\pi abcn_{B}^{3/2}+O\left(n_{B}^{\theta}\right)\right)f\left(n_{B}\right)\label{eq:sum}\\
&-\left(\frac{4}{3}\pi abcn_{A}^{3/2}+O\left(n_{A}^{\theta}\right)\right)f\left(n_{A+1}\right)\nonumber \\
&-\int_{n_{A+1}}^{n_{B}}f'\left(t\right)\left(\frac{4}{3}\pi abct^{3/2}+O\left(t^{\theta}\right)\right)\mbox{d}t\nonumber \\
&=\left(\frac{4}{3}\pi abcn_{B}^{3/2}+O\left(n_{B}^{\theta}\right)\right)f\left(n_{B}\right)\nonumber\\
&-\left(\frac{4}{3}\pi abcn_{A}^{3/2}+O\left(n_{A}^{\theta}\right)\right)f\left(n_{A+1}\right)\nonumber \\
&-\frac{4}{3}\pi abcn_{B}^{3/2}f\left(n_{B}\right)+\frac{4}{3}\pi abcn_{A+1}^{3/2}f\left(n_{A+1}\right)\nonumber \\
&+\frac{4}{3}\pi abc\cdot\frac{3}{2}\int_{n_{A+1}}^{n_{B}}f\left(t\right)t^{1/2}\mbox{d}t-\int_{n_{A+1}}^{n_{B}}f'\left(t\right)O\left(t^{\theta}\right)\mbox{d}t\nonumber \\
&=2\pi abc\int_{n_{A+1}}^{n_{B}}f\left(t\right)t^{1/2}\mbox{d}t \nonumber\\
&+O\left(n_{B}^{\theta}f\left(n_{B}\right)+n_{A}^{\theta}f\left(n_{A+1})\right)\right)\nonumber \\
&+O\left(\int_{n_{A+1}}^{n_{B}}\left|f'\left(t\right)\right|t^{\theta}\mbox{d}t\right)\nonumber.
\end{align}
Now
\[
\sum_{\left|\left|\xi\right|^{2}-\lambda\right|\geq L}\frac{1}{\left(\left|\xi\right|^{2}-\lambda\right)^{2}}=\frac{1}{\lambda^{2}}+\sum_{0<\left|\xi\right|^{2}\leq\lambda-L}\frac{1}{\left(\left|\xi\right|^{2}-\lambda\right)^{2}}+\sum_{\lambda+L\leq\left|\xi\right|^{2}}\frac{1}{\left(\left|\xi\right|^{2}-\lambda\right)^{2}}.
\]
Applying \eqref{eq:sum} with $f\left(t\right)=1/\left(t-\lambda\right)^{2},$
once with $n_{A}=n_{0}=0$ and $n_{B}\leq\lambda-L<n_{B+1}$ and then
with $n_{A}<\lambda+L\leq n_{A+1}$ and $n_{B}=\infty$ gives
\begin{align*}
\sum_{n_{0}<\left|\xi\right|^{2}\leq\lambda-L}\frac{1}{\left(\left|\xi\right|^{2}-\lambda\right)^{2}}&=2\pi abc\int_{n_{1}}^{n_{B}}\frac{t^{1/2}}{\left(t-\lambda\right)^{2}}\mbox{d}t\\
&+O\left(\frac{n_{B}^{\theta}}{\left(n_{B}-\lambda\right)^{2}}\right)+O\left(\int_{n_{1}}^{n_{B}}\frac{t^{\theta}}{\left(\lambda-t\right)^{3}}\mbox{d}t\right).
\end{align*}
Note that
\begin{align*}
\int_{n_{1}}^{n_{B}}\frac{t^{1/2}}{\left(t-\lambda\right)^{2}}\mbox{d}t\leq n_{B}^{1/2}\int_{n_{1}}^{n_{B}}\frac{1}{\left(t-\lambda\right)^{2}}\mbox{d}t&\leq\lambda^{1/2}\int_{n_{1}}^{n_{B}}\frac{1}{\left(t-\lambda\right)^{2}}\mbox{d}t\\
&=\lambda^{1/2}\left(\frac{1}{\lambda-n_{B}}-\frac{1}{\lambda-n_{1}}\right)\\
&\leq\frac{\lambda^{1/2}}{L}\leq\frac{\lambda^{\theta}}{L^{2}}
\end{align*}
also
\[
\frac{n_{B}^{\theta}}{\left(n_{B}-\lambda\right)^{2}}\leq\frac{\lambda^{\theta}}{L^{2}}
\]
and
\begin{align*}
\int_{n_{1}}^{n_{B}}\frac{t^{\theta}}{\left(\lambda-t\right)^{3}}\mbox{d}t\leq n_{B}^{\theta}\int_{n_{1}}^{n_{B}}\frac{1}{\left(\lambda-t\right)^{3}}\mbox{d}t&\leq\lambda^{\theta}\int_{n_{1}}^{n_{B}}\frac{1}{\left(\lambda-t\right)^{3}}\mbox{d}t\\
&=\frac{\lambda^{\theta}}{2}\left(\frac{1}{\left(\lambda-n_{B}\right)^{2}}-\frac{1}{\left(\lambda-n_{1}\right)^{2}}\right)\\
&\leq\frac{\lambda^{\theta}}{2L^{2}},
\end{align*}
so
\[
\sum_{n_{0}<\left|\xi\right|^{2}\leq\lambda-L}\frac{1}{\left(\left|\xi\right|^{2}-\lambda\right)^{2}}\ll\frac{\lambda^{\theta}}{L^{2}}.
\]
For the second sum, we have
\begin{align*}
\sum_{\lambda+L\leq\left|\xi\right|^{2}}\frac{1}{\left(\left|\xi\right|^{2}-\lambda\right)^{2}}&=2\pi abc\int_{n_{A+1}}^{\infty}\frac{t^{1/2}}{\left(t-\lambda\right)^{2}}\mbox{d}t\\
&+O\left(\frac{n_{A}^{\theta}}{\left(n_{A+1}-\lambda\right)^{2}}\right)+O\left(\int_{n_{A+1}}^{\infty}\frac{t^{\theta}}{\left(t-\lambda\right)^{3}}\mbox{d}t\right).
\end{align*}
Note that
\begin{align*}
\int_{n_{A+1}}^{\infty}\frac{t^{1/2}}{\left(t-\lambda\right)^{2}}\mbox{d}t&=\int_{n_{A+1}-\lambda}^{\infty}\frac{\left(s+\lambda\right)^{1/2}}{s^{2}}\mbox{d}s\\
&\leq\int_{L}^{\infty}\frac{\left(s+\lambda\right)^{1/2}}{s^{2}}\mbox{d}s\\
&=\int_{L}^{\lambda}\frac{\left(s+\lambda\right)^{1/2}}{s^{2}}\mbox{d}s+\int_{\lambda}^{\infty}\frac{\left(s+\lambda\right)^{1/2}}{s^{2}}\mbox{d}s\\
&\ll\lambda^{1/2}\int_{L}^{\lambda}\frac{1}{s^{2}}\mbox{d}s+\int_{\lambda}^{\infty}\frac{1}{s^{3/2}}\mbox{d}s\\
&\ll\frac{\lambda^{1/2}}{L}\leq\frac{\lambda^{\theta}}{L^{2}}
\end{align*}
also
\[
\frac{n_{A}^{\theta}}{\left(n_{A+1}-\lambda\right)^{2}}\ll\frac{\lambda^{\theta}}{L^{2}}
\]
and
\begin{align*}
\int_{n_{A+1}}^{\infty}\frac{t^{\theta}}{\left(t-\lambda\right)^{3}}\mbox{d}t&=\int_{n_{A+1}-\lambda}^{\infty}\frac{\left(s+\lambda\right)^{\theta}}{s^{3}}\mbox{d}s\\
&\leq\int_{L}^{\infty}\frac{\left(s+\lambda\right)^{\theta}}{s^{3}}\mbox{d}s\\
&=\int_{L}^{\lambda}\frac{\left(s+\lambda\right)^{\theta}}{s^{3}}\mbox{d}s+\int_{\lambda}^{\infty}\frac{\left(s+\lambda\right)^{\theta}}{s^{3}}\mbox{d}s\\
&\ll\lambda^{\theta}\int_{L}^{\lambda}\frac{1}{s^{3}}\mbox{d}s+\int_{\lambda}^{\infty}\frac{1}{s^{3-\theta}}\mbox{d}s\ll\frac{\lambda^{\theta}}{L^{2}}
\end{align*}
so
\[
\sum_{\lambda+L\leq\left|\xi\right|^{2}}\frac{1}{\left(\left|\xi\right|^{2}-\lambda\right)^{2}}\ll\frac{\lambda^{\theta}}{L^{2}}.
\]
We conclude that
\[
\left\|G_{\lambda}-G_{\lambda,L}\right\|^{2}\ll\frac{\lambda^{\theta}}{L^{2}}
\]
and therefore by Lemma \ref{lem:IrrationalNormBound} 
\[
\left\|g_{\lambda}-g_{\lambda,L}\right\|_{2}^{2}\ll\frac{\lambda^{-1+\theta+2\varepsilon}}{L^{2}}=\lambda^{2\delta-1+\theta+2\varepsilon}.
\]
Since $\delta<\left(1-\theta\right)/2-\varepsilon$ this tends to
zero.
\end{proof}
From this, as in Lemmas \ref{lem:RationalNormEquiv}, \ref{lem:RationalSuffConv},
we get the next two lemmas:
\begin{lem}
Let $0<\delta<\min\left\{ \left(1-\theta\right)/2-\varepsilon,1/\tau-\varepsilon\right\} ,\, L=\lambda^{-\delta}$,
we have
\[
\left\|G_{\lambda,L}\right\|_{2}=\left\|G_{\lambda}\right\|_{2}\left(1+o\left(1\right)\right)
\]
as $\lambda\to\infty$ along $\Lambda_{2}$.
\end{lem}
and
\begin{lem}
\label{lem:IrrationalConvSuff}Let $f\in C^{\infty}\left(\mathbb{T}^{3}\right)$
and $0<\delta<\min\left\{ \left(1-\theta\right)/2-\varepsilon,1/\tau-\varepsilon\right\} $,
$L=\lambda^{-\delta},$ we have
\[
\left|\left\langle fg_{\lambda},g_{\lambda}\right\rangle -\left\langle fg_{\lambda,L},g_{\lambda,L}\right\rangle \right|\to0
\]
as $\lambda\to\infty$ along $\Lambda_{2}$, so
\[
\left\langle fg_{\lambda,L},g_{\lambda,L}\right\rangle \to0\,\Rightarrow\left\langle fg_{\lambda},g_{\lambda}\right\rangle \to0
\]
as $\lambda\to\infty$ along $\Lambda_{2}$. 
\end{lem}

\subsection{A Density One Set}

Let $0\neq\zeta\in\mathcal{L}$, and denote $\zeta=\left(\zeta_{1}/a,\zeta_{2}/b,\zeta_{3}/c\right)$.
Assume that $\zeta_{3}\neq0$ (the other cases are symmetric).

Define
\[
S_{\zeta}=\left\{ \xi\in\mathcal{L}:\,\left|\xi\right|^{2}\geq1,\,\left|2\left\langle \xi,\zeta\right\rangle -\left|\zeta\right|^{2}\right|<1/4c^{2}\right\} .
\]
We prove now a simple upper bound for the number of elements in $S_{\zeta}$
up to $X$:
\begin{lem}
We have 
\[
\#\left\{ \xi\in S_{\zeta}:\left|\xi\right|^{2}\leq X\right\} \ll X.
\]
\end{lem}
\begin{proof}
For every $\xi\in S_{\zeta}$ such that $\left|\xi\right|^{2}\leq X$,
denote $\xi=\left(\xi_{1}/a,\xi_{2}/b,\xi_{3}/c\right).$

For $1\leq i\leq3$ we have $\left|\xi_{i}\right|\ll X^{1/2}$, and
\[
\left|2\left\langle \xi,\zeta\right\rangle -\left|\zeta\right|^{2}\right|=\left|\frac{\zeta_{1}}{a^{2}}\left(2\xi_{1}-\zeta_{1}\right)+\frac{\zeta_{2}}{b^{2}}\left(2\xi_{2}-\zeta_{2}\right)+\frac{\zeta_{3}}{c^{2}}\left(2\xi_{3}-\zeta_{3}\right)\right|<1/4c^{2}
\]
so
\[
\left|\frac{c^{2}}{a^{2}}\zeta_{1}\left(2\xi_{1}-\zeta_{1}\right)+\frac{c^{2}}{b^{2}}\zeta_{2}\left(2\xi_{2}-\zeta_{2}\right)+\zeta_{3}\left(2\xi_{3}-\zeta_{3}\right)\right|<1/4
\]
and we see that $\zeta_{3}\left(2\xi_{3}-\zeta_{3}\right)$ is uniquely
determined by the values of $\xi_{1},\xi_{2}$, and since $\zeta_{3}\neq0,$
we get that $\xi_{3}$ is uniquely determined by the values of $\xi_{1},\xi_{2}.$
Since $\left|\xi_{1}\right|,\left|\xi_{2}\right|\ll X^{1/2}$, we
conclude that
\[
\#\left\{ \xi\in S_{\zeta}:\left|\xi\right|^{2}\leq X\right\} \ll X
\]
as claimed.
\end{proof}
For $0\neq\zeta\in\mathcal{L},$ define 
\[
\Lambda_{\zeta}=\Lambda_{\zeta}^{\varepsilon,\delta}=\left\{ \lambda\in\Lambda_{2}:\, A\left(\lambda,L\right)\cap S_{\zeta}=\emptyset\right\} 
\]
(recall that $ L = \lambda^{-\delta} $).

We claim that this is a density one set in $\Lambda_{2}$ (and hence
in $\Lambda)$, i.e. if we denote
\[
B_{\zeta}=B_{\zeta}^{\varepsilon,\delta}=\Lambda_{2}\setminus\Lambda_{\zeta}=\left\{ \lambda\in\Lambda_{2}:\, A\left(\lambda,L\right)\cap S_{\zeta}\neq\emptyset\right\} 
\]
then
\begin{lem}
We have $\left\{ \lambda\in B_{\zeta}:\,\lambda\leq X\right\} =o\left(X^{3/2}\right)$.\end{lem}
\begin{proof}
Define $\mathcal{N}_{\zeta}\subseteq\mathcal{N}$ to be the set of
norms $\left|\xi\right|^{2}$ of $\xi\in S_{\zeta}$. We have 
\[
\#\left\{ n\in\mathcal{N}_{\zeta}:\, n\leq2X\right\} \asymp\#\left\{ \xi\in S_{\zeta}:\,\left|\xi\right|^{2}\leq2X\right\} \ll X.
\]
We have a map $\iota:B_{\zeta}\to\mathcal{N}_{\zeta}$ defined by
$\iota\left(\lambda\right)$ being the closest element $n\in\mathcal{N}_{\zeta}$
to $\lambda$; if there are two such elements, i.e. $n_{-}<\lambda<n_{+}$
with $n_{\pm}\in\mathcal{N}_{\zeta}$ and $n_{+}-\lambda=\lambda-n_{-},$
then set $\iota\left(\lambda\right)=n_{+}.$

We have $\#\iota^{-1}\left(n\right)\ll n^{1/2-\delta+2\varepsilon}.$
Indeed, first note that 
\[
\iota^{-1}\left(n\right)\leq\#\left\{ \lambda\in\Lambda_{2}:\,\exists\xi\in S_{\zeta}\cap A\left(\lambda,L\right),\,\left|\xi\right|^{2}=n\right\} .
\]
For every $\lambda$ in the set above we have $\left|\lambda-n\right|\leq\lambda^{-\delta},$
and if we choose one of the elements in the set, say $\lambda_{0}$,
we have that for every other $\lambda$ (assuming that $n$ is large
enough): $\left|\lambda_{0}-\lambda\right|\leq\lambda_{0}^{-\delta}+\lambda^{-\delta}\leq3\lambda_{0}^{-\delta}$
(since for large enough $n$: $\lambda\geq n-\lambda^{-\delta}\geq\lambda_{0}-\lambda_{0}^{-\delta}-\lambda^{-\delta}\geq\lambda_{0}/2$),
so 
\begin{align*}
\#\left\{ \lambda\in\Lambda_{2}:\,\exists\xi\in S_{\zeta}\cap A\left(\lambda,L\right),\,\left|\xi\right|^{2}=n\right\}&\ll\#\tilde{A}\left(\lambda_{0},3\lambda_{0}^{-\delta}\right)\\
&\leq\lambda_{0}^{1/2-\delta+2\varepsilon}\\
&\ll n^{1/2-\delta+2\varepsilon}.
\end{align*}
We conclude that
\begin{align*}
\#\left\{ \lambda\in B_{\zeta}:\lambda\leq X\right\} &\leq\sum_{\begin{subarray}{c}
n\in\mathcal{N}_{\zeta}\\
n\leq2X
\end{subarray}}\iota^{-1}\left(n\right)\\
&\ll X^{1/2-\delta+2\varepsilon}\cdot\#\left\{ n\in\mathcal{N}_{\zeta}:\, n\leq2X\right\}\\
&\ll X^{3/2-\delta+2\varepsilon}
\end{align*}
proving our claim.
\end{proof}
From the last lemma we conclude that:
\begin{lem}
We have $\left\langle e_{\zeta}g_{\lambda,L},g_{\lambda,L}\right\rangle \to0$
as $\lambda\to\infty$ along $\lambda\in\Lambda_{\zeta}.$\end{lem}
\begin{proof}
We have
\[
\left\langle e_{\zeta}G_{\lambda,L},G_{\lambda,L}\right\rangle \asymp\sum_{\xi\in A\left(\lambda,L\right)}\frac{1}{\left(\left|\xi\right|^{2}-\lambda\right)\left(\left|\xi-\zeta\right|^{2}-\lambda\right)}.
\]
Note that
\[
\left|\left|\xi-\zeta\right|^{2}-\lambda\right|=\left|\left|\xi\right|^{2}-\lambda-2\left\langle \xi,\zeta\right\rangle +\left|\zeta\right|^{2}\right|\geq\left|2\left\langle \xi,\zeta\right\rangle -\left|\zeta\right|^{2}\right|-\left|\left|\xi\right|^{2}-\lambda\right|
\]
and since $\lambda\in\Lambda_{\zeta}$ we have $S_{\zeta}\cap A\left(\lambda,L\right)=\emptyset$,
so 
\[
\left|2\left\langle \xi,\zeta\right\rangle -\left|\zeta\right|^{2}\right|\geq\frac{1}{4c^{2}}.
\]
We get that for large enough $\lambda$ 
\[
\left|\left|\xi-\zeta\right|^{2}-\lambda\right|\geq\frac{1}{4c^{2}}-L\gg1
\]
and therefore
\begin{align*}
\left|\left\langle e_{\zeta}G_{\lambda,L},G_{\lambda,L}\right\rangle \right|&\ll\sum_{\xi\in A\left(\lambda,L\right)}\frac{1}{\left|\left(\left|\xi\right|^{2}-\lambda\right)\right|}\\
&\ll\left\|G_{\lambda}\right\|_{2}\left(\sum_{\xi\in A\left(\lambda,L\right)}1\right)^{1/2}.
\end{align*}
Since $\lambda\in\Lambda_{\zeta}\subseteq\Lambda_{2}$, and since
the multiplicities $r_{\mathcal{L}}\left(n\right)$ are bounded, we
have
\[
\sum_{\xi\in A\left(\lambda,L\right)}1=\#A\left(\lambda,L\right)\ll\#\tilde{A}\left(\lambda,3L\right)\leq L\lambda^{1/2+2\varepsilon}
\]
so 
\[
\left|\left\langle e_{\zeta}g_{\lambda,L},g_{\lambda,L}\right\rangle \right|\ll\frac{\left\|G_{\lambda}\right\|_{2}\lambda^{1/4-\delta/2+\varepsilon}}{\left\|G_{\lambda,L}\right\|_{2}^{2}}\ll\frac{\lambda^{1/4-\delta/2+\varepsilon}}{\left\|G_{\lambda}\right\|_{2}}\ll\lambda^{-1/4-\delta/2+2\varepsilon}
\]
which tends to zero as $\lambda\to\infty$ (for $\varepsilon>0$ small
enough).
\end{proof}
We conclude from Lemma \ref{lem:IrrationalConvSuff} that
\[
\left\langle e_{\zeta}g_{\lambda},g_{\lambda}\right\rangle \to0
\]
as $\lambda\to\infty$ along $\lambda\in\Lambda_{\zeta}.$

\subsection{Proof of Theorem \ref{thm:Main2}}

We now use a diagonalization argument to prove Theorem \ref{thm:Main2}:
\begin{thm*}
There is a density one subset $\Lambda_{\infty}\subseteq\Lambda$
so that for every observable $a\in C^{\infty}\left(\mathbb{T}^{3}\right),$
we have
\[
\left\langle ag_{\lambda},g_{\lambda}\right\rangle \to\frac{1}{\mbox{area}\left(\mathbb{T}^{3}\right)}\int_{\mathbb{T}^{3}}a\left(x\right)\mbox{d}x
\]
 as $\lambda\to\infty$ along $\Lambda_{\infty}$.\end{thm*}
\begin{proof}
For $J\geq1$, let $\Lambda_{J}\subseteq\Lambda$ be of density one
so that for all $\left|\zeta\right|\leq J$, $\left\langle e_{\zeta}g_{\lambda},g_{\lambda}\right\rangle \to0$
as $\lambda\to\infty$ along $\Lambda_{J},$ and in particular for
every trigonometric polynomial $P_{J}\left(x\right)=\sum\limits _{\left|\zeta\right|\leq J}p_{\zeta}e_{\zeta}\left(x\right)$
we have
\begin{equation}
\left\langle P_{J}g_{\lambda},g_{\lambda}\right\rangle \to\frac{1}{\mbox{area}\left(\mathbb{T}^{3}\right)}\int_{\mathbb{T}^{3}}P_{J}\left(x\right)\mbox{d}x\label{eq:trig}
\end{equation}
along $\Lambda_{J}$. 

We can assume that $\Lambda_{J+1}\subseteq\Lambda_{J}$ for each $J$.
Choose $M_{J}$ so that $M_{J}\uparrow\infty$ as $J\to\infty$, and
so that for all $X>M_{J}$
\[
\frac{\#\left\{ \lambda\in\Lambda_{J}:\,\lambda\leq X\right\} }{\#\left\{ \lambda\in\Lambda:\,\lambda\leq X\right\} }\geq1-\frac{1}{2^{J}}
\]
and let $\Lambda_{\infty}$ be such that $\Lambda_{\infty}\cap\left[M_{J},M_{J+1}\right]=\Lambda_{J}\cap\left[M_{J},M_{J+1}\right]$
for all $J$. Then $\Lambda_{\infty}\cap\left[0,M_{J+1}\right]$ contains
$\Lambda_{J}\cap\left[0,M_{J+1}\right]$ and therefore $\Lambda_{\infty}$
has density one in $\Lambda$, and \eqref{eq:trig} holds for $\lambda\in\Lambda_{\infty}.$ 

The theorem now follows from the density of trigonometric polynomials
in $C^{\infty}\left(\mathbb{T}^{3}\right)$ in the uniform norm (as
in the proof of Theorem \ref{thm:Main1}).
\end{proof}
\appendix

\section{~}

\subsection{Integral Points in Spherical Strips}

We estimate the number of integral points inside some strips on 3D 
spheres. The strategy is to estimate the integral points on every
circle in the strip: a simple substitution reduces the problem to
counting integral points on two-dimensional ellipses, which could
be treated by some basic algebraic number theory.
\begin{lem}
\label{lem:IntSphericalStripsEst}Let $L=\lambda^{\delta},\,0<\delta<1/4$.
For every $0\neq\zeta\in\mathbb{Z}^{3},C_{1},C_{2}$ and $n$ such
that $\left|n-\lambda\right|<C_{1}L$, we have 
\[
\#\left\{ \eta\in\mathbb{Z}^{3}:\,\left|\eta\right|^{2}=n,\,\left|\left\langle \eta,\zeta\right\rangle \right|<C_{2}L\right\} \ll_{C_{1},C_{2},\zeta,\varepsilon}Ln^{\varepsilon}.
\]
\end{lem}
\begin{proof}
Denote $\zeta=\left(\zeta_{1},\zeta_{2},\zeta_{3}\right)$, and without
loss of generality we can assume that $\zeta_{3}\neq0$. For $\eta=\left(x,y,z\right)$
with $\left|\eta\right|^{2}=n$, $\left\langle \eta,\zeta\right\rangle =m$,
$\left|m\right|<C_{2}L$, we have $x^{2}+y^{2}+z^{2}=n$, $\zeta_{1}x+\zeta_{2}y+\zeta_{3}z=m$.
since $\zeta_{3}\neq0$ we get $z=\frac{m-\zeta_{1}x-\zeta_{2}y}{\zeta_{3}}$,
and substitution gives:
\begin{equation}
ax^{2}+2bxy+cy^{2}+2dx+2ey+f=0\label{eq:RationalStrip1}
\end{equation}
where
\begin{alignat*}{1}
a= & \zeta_{1}^{2}+\zeta_{3}^{2}\\
b= & \zeta_{1}\zeta_{2}\\
c= & \zeta_{2}^{2}+\zeta_{3}^{2}\\
d= & -\zeta_{1}m\\
e= & -\zeta_{2}m\\
f= & -\zeta_{3}^{2}n+m^{2}.
\end{alignat*}
Note that $c>0,\, ac-b^{2}=\zeta_{3}^{2}\left(\zeta_{1}^{2}+\zeta_{2}^{2}+\zeta_{3}^{2}\right)>0$.

Completing the square using the fact that $c\neq0$, we get:
\[
\left(ac-b^{2}\right)x^{2}+\left(bx+cy\right)^{2}+2cdx+2cey+cf=0.
\]
Setting $y'=bx+cy$ (and then $y=\frac{y'-bx}{c})$, we get that the
number of integer solutions to equation \eqref{eq:RationalStrip1}
is bounded by the number of integer solutions to the equation 
\begin{equation}
\left(ac-b^{2}\right)x^{2}+y^{2}+2\left(cd-be\right)x+2ey+cf=0.\label{eq:RationalStrip2}
\end{equation}
Completing the square again, we get
\[
\left(ac-b^{2}\right)x^{2}+\left(y+e\right)^{2}+2\left(cd-be\right)x+cf-e^{2}=0.
\]
Setting $y'=y+e$, we get that the number of integer solutions to
equation \eqref{eq:RationalStrip2} is equal to the number of integer
solutions to
\begin{equation}
\left(ac-b^{2}\right)x^{2}+y^{2}+2\left(cd-be\right)x+cf-e^{2}=0.\label{eq:RationalStrip3}
\end{equation}
Completing the square for the last time using the fact that $ac-b^{2}\neq0$,
we get
\begin{align*}
\left(\left(ac-b^{2}\right)x+\left(cd-be\right)\right)^{2}&+\left(ac-b^{2}\right)y^{2}\\
&+\left(ac-b^{2}\right)\left(cf-e^{2}\right)-\left(cd-be\right)^{2}=0.
\end{align*}
Setting $x'=\left(ac-b^{2}\right)x+\left(cd-be\right)$, we get that
the number of integer solutions to equation \eqref{eq:RationalStrip3}
is bounded by the number of integer solutions to
\begin{equation}
x^{2}+\left(ac-b^{2}\right)y^{2}=\left(ac-b^{2}\right)\left(-cf+e^{2}\right)+\left(cd-be\right)^{2}.\label{eq:RationalStrip4}
\end{equation}
Denote $ac-b^{2}=t^{2}D,$ where $D>0$ is squarefree, and 
\[
k=\left(ac-b^{2}\right)\left(-cf+e^{2}\right)+\left(cd-be\right)^{2}.
\]
We get
\[
x^{2}+D\left(ty\right)^{2}=k
\]
and setting $y'=ty,$ we get that the number of integer solutions
to equation \eqref{eq:RationalStrip4} is bounded by the number of
integer solutions to
\[
x^{2}+Dy^{2}=k
\]
i.e. by the number $r_{D}\left(k\right)$ of representations of an
integer $k$ by the quadratic form $x^{2}+Dy^{2}$. Now we claim that
\begin{equation}
r_{D}\left(k\right)\leq6\tau\left(k\right)\label{eq:RationalStrip5}
\end{equation}
where $\tau\left(k\right)$ is the number of divisors of $k$. Since
\[
\tau\left(k\right)\ll_{\varepsilon}k^{\varepsilon}\ll_{C_{1},C_{2},\zeta,\varepsilon}n^{\varepsilon}
\]
we conclude that 
\[
\#\left\{ \eta\in\mathbb{Z}^{3}:\,\left|\eta\right|^{2}=n,\,\left|\left\langle \eta,\zeta\right\rangle \right|<C_{2}L\right\} \ll_{C_{1},C_{2},\zeta,\varepsilon}Ln^{\varepsilon}
\]
as claimed. 

The estimate \eqref{eq:RationalStrip5} follows from factorization
into prime ideals in the ring of integers $\mathbb{A}$ of the imaginary
quadratic extension $\mathbb{Q}\left(\sqrt{-D}\right)$. Indeed, given
any prime $p$, consider the principal ideal $\left\langle p\right\rangle $
in $\mathbb{A}$. Then (cf. \cite{Mollin}) either $\left\langle p\right\rangle $
is a prime ideal, or $\left\langle p\right\rangle =\mathcal{P}_{1}\mathcal{P}_{2}$,
where $\mathcal{P}_{1},\mathcal{P}_{2}$ are prime ideals (and not
necessarily different). The fundamental theorem of arithmetic yields
\[
k=\prod_{\left\langle q_{j}\right\rangle \, is\, prime}q_{j}^{\beta_{j}}\prod_{\left\langle p_{l}\right\rangle =\mathcal{P}_{l,1}\mathcal{P}_{l,2}}p_{l}^{\alpha_{l}}
\]
so we get the unique factorization
\[
\left\langle k\right\rangle =\prod\left\langle q_{j}\right\rangle ^{\beta_{j}}\prod\mathcal{P}_{l,1}^{\alpha_{l}}\mathcal{P}_{l,2}^{\alpha_{l}}.
\]
Each representation of $k=x^{2}+Dy^{2}$ corresponds to a decomposition
of the principal ideal 
\[
\left\langle k\right\rangle =\left\langle x+y\sqrt{-D}\right\rangle \left\langle x-y\sqrt{-D}\right\rangle 
\]
where $N\left(\left\langle x+y\sqrt{-D}\right\rangle \right)=N\left(\langle x-y\sqrt{-D}\rangle\right)=k$,
so from the uniqueness of the factorization we get that
\begin{align*}
\left\langle x+y\sqrt{-D}\right\rangle =\prod\left\langle q_{j}\right\rangle ^{\beta_{j}/2}\prod\mathcal{P}_{l,1}^{\gamma_{l}}\mathcal{P}_{l,2}^{\alpha_{l}-\gamma_{l}}\\
\left\langle x-y\sqrt{-D}\right\rangle =\prod\left\langle q_{j}\right\rangle ^{\beta_{j}/2}\prod\mathcal{P}_{l,1}^{\alpha_{l}-\gamma_{l}}\mathcal{P}_{l,2}^{\gamma_{l}}
\end{align*}
where $0\leq\gamma_{l}\leq\alpha_{l}$. From this follows that the
number of possibilities for $\left\langle x+y\sqrt{-D}\right\rangle $
is bounded by $\prod\left(1+\alpha_{l}\right)\leq\tau\left(k\right)$.
But $\left\langle x+y\sqrt{-D}\right\rangle =\left\langle x'+y'\sqrt{-D}\right\rangle $
if and only if $x+y\sqrt{-D}$ and $x'+y'\sqrt{-D}$ are associates
in $\mathbb{A}$, and since the number of units in $\mathbb{A}$ is
at most $6$, we get that $r_{D}\left(k\right)\leq6\tau\left(k\right)$.
\end{proof}

\subsection{\label{sub:AppDiscrepancy}Discrepancy}

We prove some results from the theory of uniform distribution modulo
$1$. Most of this section is adapted from \cite{Kuipers}.

Let us recall the definition of discrepancy:
\begin{defn}
Let $\left(x_{n}\right)$ be a sequence of real numbers. The number
\[
D_{N}=\sup_{0\leq a<b\leq1}\left|\frac{\#\left\{ n\leq N:\,\left\{ x_{n}\right\} \in[a,b)\right\} }{N}-\left(b-a\right)\right|
\]
is called the discrepancy of the given sequence.
\end{defn}
A useful upper bound for the discrepancy is given by the theorem of
Erd\H{o}s-Turán \cite{Erdos}:
\begin{thm}
There exists an absolute constant $C$, such that for any real numbers
$x_{1},\dots,x_{N}$ and for any positive integer $m$, we have
\[
D_{N}\leq C\left(\frac{1}{m}+\sum_{h=1}^{m}\frac{1}{h}\left|\frac{1}{N}\sum_{n=1}^{N}e^{2\pi ihx_{n}}\right|\right).
\]

\end{thm}
Let $\alpha\in\mathbb{R}$ be an irrational of finite type $\tau$,
and let $\beta\in\mathbb{R}$. Define $x_{n}=\alpha n+\beta.$ By
the theorem of Erd\H{o}s-Turán, the discrepancy of $\left(x_{n}\right)$
is bounded for any positive integer $m$ by
\begin{align}
D_{N}&\ll\frac{1}{m}+\frac{1}{N}\sum_{h=1}^{m}\frac{1}{h}\left|\sum_{n=1}^{N}e^{2\pi ih\left(n\alpha+\beta\right)}\right|\nonumber \\
&=\frac{1}{m}+\frac{1}{N}\sum_{h=1}^{m}\frac{1}{h}\left|\sum_{n=1}^{N}e^{2\pi ihn\alpha}\right|\nonumber \\
&=\frac{1}{m}+\frac{1}{N}\sum_{h=1}^{m}\frac{1}{h}\frac{\left|1-e^{2\pi ihN\alpha}\right|}{\left|1-e^{2\pi ih\alpha}\right|}\nonumber\\
&\leq\frac{1}{m}+\frac{1}{N}\sum_{h=1}^{m}\frac{1}{h}\frac{2}{\left|1-e^{2\pi ih\alpha}\right|}\label{eq:DiscEst}\\
&=\frac{1}{m}+\frac{1}{N}\sum_{h=1}^{m}\frac{1}{h}\frac{1}{\left|\sin\pi h\alpha\right|}\nonumber\\
&\leq\frac{1}{m}+\frac{1}{2N}\sum_{h=1}^{m}\frac{1}{h}\frac{1}{\left\|h\alpha\right\|}\nonumber.
\end{align}
From here we could continue by
\[
\frac{1}{m}+\frac{1}{2N}\sum_{h=1}^{m}\frac{1}{h}\frac{1}{\left\|h\alpha\right\|}\ll\frac{1}{m}+\frac{1}{2N}\sum_{h=1}^{m}h^{\tau+\varepsilon-1}\ll\frac{1}{m}+\frac{m^{\tau+\varepsilon}}{N}
\]
and choosing $m=\lfloor N^{1/\left(1+\tau\right)}\rfloor$ we could
deduce that $D_{N}=O\left(N^{-1/\left(1+\tau\right)+\varepsilon}\right)$.
But with a very little effort we can get a better estimate for the
the RHS of \eqref{eq:DiscEst}:
\begin{lem}
\label{lem:OneOverDistEst}Let $\alpha$ be an irrational of finite
type $\tau$, and let $m$ be a positive integer. Then for every $\epsilon>0$,
there exists a positive constant $c=c\left(\alpha,\varepsilon\right)$
such that 
\[
\sum_{h=1}^{m}\frac{1}{\left\|h\alpha\right\|}\leq cm^{\tau+\varepsilon}.
\]
\end{lem}
\begin{proof}
For $0\leq q_{1}<q_{2}\leq m$, we have
\[
\left\|q_{2}\alpha\pm q_{1}\alpha\right\|=\left\|\left(q_{2}\pm q_{1}\right)\alpha\right\|\geq\frac{c}{\left(q_{2}\pm q_{1}\right)^{\tau+\varepsilon/2}}\geq\frac{c}{\left(2m\right)^{\tau+\varepsilon/2}}.
\]
But $\left\|q_{1}\alpha\right\|=\left|q_{1}\alpha-n_{1}\right|$
for some integer $n_{1}$, and $\left\|q_{2}\alpha\right\|=\left|q_{2}\alpha-n_{2}\right|$
for some integer $n_{2}$, and hence
\begin{align*}
\left|\left\|q_{2}\alpha\right\|-\left\|q_{1}\alpha\right\|\right|&=\left|\left|q_{2}\alpha-n_{2}\right|-\left|q_{1}\alpha-n_{1}\right|\right|\\
&\geq\left\|q_{2}\alpha\pm q_{1}\alpha\right\|\\
&\geq\frac{c}{\left(2m\right)^{\tau+\varepsilon/2}}.
\end{align*}
This implies that in each of the intervals
\[
[0,\frac{c}{\left(2m\right)^{\tau+\varepsilon/2}}),\,[\frac{c}{\left(2m\right)^{\tau+\varepsilon/2}},\frac{2c}{\left(2m\right)^{\tau+\varepsilon/2}}),\,\dots,\,[\frac{mc}{\left(2m\right)^{\tau+\varepsilon/2}},\frac{\left(m+1\right)c}{\left(2m\right)^{\tau+\varepsilon/2}})
\]
there is at most one number of the form $\left\|h\alpha\right\|,$
$1\leq h\leq m$, with no such number in the first interval. Therefore
\[
\sum_{h=1}^{m}\frac{1}{\left\|h\alpha\right\|}\leq\sum_{h=1}^{m}\frac{\left(2m\right)^{\tau+\varepsilon/2}}{hc}=\frac{\left(2m\right)^{\tau+\varepsilon/2}}{c}\sum_{h=1}^{m}\frac{1}{h}\leq\tilde{c}m^{\tau+\varepsilon}.\qedhere
\]
\end{proof}
\begin{cor}
\label{cor:OneOverDistSmooth}Let $\alpha$ be an irrational of finite
type $\tau$, and let $m$ be a positive integer. Then for every $\epsilon>0$,
there exists a positive constant $c=c\left(\alpha,\varepsilon\right)$
such that
\[
\sum_{h=1}^{m}\frac{1}{h\left\|h\alpha\right\|}\leq cm^{\tau-1+\varepsilon}.
\]
\end{cor}
\begin{proof}
Define $S\left(t\right)=\sum\limits _{h\leq t}\frac{1}{\left\|h\alpha\right\|}$.

From Lemma \ref{lem:OneOverDistEst} we have $S\left(t\right)=S\left(\lfloor t\rfloor\right)\leq c\lfloor t\rfloor^{\tau+\varepsilon}\leq ct^{\tau+\varepsilon}.$
Using partial summation we conclude that
\begin{align*}
\sum_{h=1}^{m}\frac{1}{h\left\|h\alpha\right\|}&=\frac{S\left(m\right)}{m}+\int_{1}^{m}\frac{S\left(t\right)}{t^{2}}\mbox{d}t\\
&\leq cm^{\tau-1+\varepsilon}+c\int_{1}^{m}\frac{\mbox{d}t}{t^{2-\tau-\varepsilon}}\\
&=c\left(m^{\tau-1+\varepsilon}+\frac{m^{\tau-1+\varepsilon}-1}{\tau-1+\varepsilon}\right)\\
&\leq\tilde{c}m^{\tau-1+\varepsilon}.\qedhere
\end{align*}

\end{proof}
Substituting Corollary \ref{cor:OneOverDistSmooth} in \eqref{eq:DiscEst},
we conclude that $D_{N}\ll_{\alpha,\varepsilon}\frac{1}{m}+\frac{m^{\tau-1+\varepsilon}}{N}$,
and choosing $m=\lfloor N^{1/\tau}\rfloor$, we get that there exists a positive constant $ c=c\left(\alpha,\varepsilon\right) $ such that
\begin{equation}
D_{N}\le cN^{-1/\tau+\varepsilon}.\label{eq:DiscrepancyEst}
\end{equation}
Note that the constant $ c $ does not depend on $ \beta $.


\begin{thebibliography}{99}
\bibitem{Colin1}Y. Colin de Verdière, \emph{Pseudo-laplaciens I},
Annales de l'Institut Fourier, tome 32, no. 3 (1982), 275-286.

\bibitem{Colin2}Y. Colin de Verdière, \emph{Ergodicité et fonctions
propres du laplacien}, Comm. Math. Phys. 102 (1985), 497-502.

\bibitem{Erdos}P. Erd\H{o}s and P. Turán, \emph{On a problem in the
theory of uniform distribution I,II}, Indag. Math 10 (1948), 370-378,
406-413.

\bibitem{Grosswald}E. Grosswald, \emph{Representations of Integers
as Sums of Squares}, Springer-Verlag, New-York, 1985.

\bibitem{Hlawka}E. Hlawka, \emph{Über Integrale auf Konvexen Körpern
I,II}, \emph{Monatsh}. Math. 54 (1950), 1-36, 81-99.

\bibitem{Kuipers}L. Kuipers and H. Niederreiter, \emph{Uniform Distribution
of Sequences}, Dover Publishing, 2006.

\bibitem{Mollin}R. A. Mollin, \emph{Algebraic Number Theory}, Second
Edition, Chapman and Hall/CRC, 2011.

\bibitem{Privat}Y. Privat, E. Trélat and E. Zuazua, \emph{On the
best observation of wave and Schrödinger equations in quantum ergodic
billiards}, Proceedings Journées EDP, Biarritz (2012). 

\bibitem{Roth}K. F. Roth, \emph{Rational approximations to algebraic
numbers}, Mathematika 2 (1955), 1-20, 168.

\bibitem{Rudnick}Z. Rudnick and H. Ueberschär, \emph{Wave function
statistics for a point scatterer on the torus}, Comm. Math. Phys. 316 (2012), 763-782.

\bibitem{Schnirelman}A. Schnirelman, \emph{Ergodic properties of
eigenfunctions}, Usp. Math. Nauk. 29 (1974), 181-182.

\bibitem{Siegel}C. L. Siegel, \emph{Über die Classenzahl quadratischer
Zahlkörper}, Acta Arith. 1 (1935), 83-86.

\bibitem{Zelditch}S. Zelditch, \emph{Uniform distribution of eigenfunctions
on compact hyperbolic surfaces}, Duke Math. J. 55 (1987), 919-941.\end{thebibliography}
\end{document}